\documentclass[a4paper,twoside,11pt]{amsart}
\usepackage[english]{babel}
\usepackage[T1]{fontenc}
\usepackage{amsmath}
\usepackage{amsthm}
\usepackage{amssymb}
\usepackage{mathrsfs}
\usepackage[linktocpage]{hyperref}
\usepackage{xcolor}
\usepackage{mathtools}
\usepackage{tikz-cd}
\usepackage{fancybox}
\usepackage{amscd}
\usepackage{float}
\usepackage{stmaryrd}
\usepackage{xspace}

\usepackage{enumitem}

\usepackage[all,matrix,arrow,cmtip]{xy} 
\SelectTips{cm}{}
\newdir{(}{{}*!/-5pt/@^{(}}
\entrymodifiers={+!!<0pt,\fontdimen22\textfont2>}

\makeatletter
\newtheorem*{rep@theorem}{\rep@title}
\newcommand{\newreptheorem}[2]{%
	\newenvironment{rep#1}[1]{%
		\def\rep@title{#2 \ref{##1}}%
		\begin{rep@theorem}}%
		{\end{rep@theorem}}}
\makeatother

\theoremstyle{plain}
\newtheorem{theorem}{Theorem}[section]
\newreptheorem{theorem}{Theorem}
\newtheorem*{theorem*}{Theorem}
\newtheorem{lemma}[theorem]{Lemma}
\newtheorem{proposition}[theorem]{Proposition}
\newtheorem{corollary}[theorem]{Corollary}
\theoremstyle{definition}
\newtheorem{definition}[theorem]{Definition}
\newtheorem*{definition*}{Definition}
\newtheorem{example}[theorem]{Example}

\theoremstyle{remark}
\newtheorem{remark}[theorem]{Remark}

\numberwithin{equation}{section}

\newcommand{\NN}{\mathbb{N}} 
\newcommand{\ZZ}{\mathbb{Z}} 
\newcommand{\CC}{\mathbb{C}} 

\newcommand{\id}{\mathrm{id}} 
\newcommand{\GG}{\mathbb{G}} 

\DeclareMathOperator{\colim}{{colim}}
\newcommand{\nc}{\mathrm{nc}}

\newcommand{\op}{\mathrm{op}}

\newcommand{\cd}{\mathrm{cd}}

\DeclareMathOperator{\chr}{char}

\newcommand\sA{\mathcal{A}}
\newcommand\sO{\mathcal{O}} 

\newcommand{\QAlg}{\mathsf{QAlg}}

\newcommand{\Space}{\mathsf{S}}

\newcommand{\Stk}{\mathsf{Stk}}
\newcommand\Mod{\mathsf{Mod}}
\newcommand\Cat{\mathsf{Cat}}

\DeclareMathOperator{\Map}{Map}

\newcommand{\Alg}{\mathsf{Alg}}

\newcommand{\Aff}{\mathsf{Aff}}

\newcommand{\ucd}{$\mathrm{ucd}_{0}$\xspace}
\renewcommand{\cd}{$\mathrm{cd}_{0}$\xspace}

\renewcommand{\AA}{\mathbb{A}} 
\newcommand{\PP}{\mathbb{P}} 

\DeclareMathOperator{\Sym}{Sym} 
\DeclareMathOperator{\Spec}{Spec} 
\DeclareMathOperator{\Proj}{Proj} 
\DeclareMathOperator{\Bl}{Bl}
\newcommand{\QCoh}{\mathsf{QCoh}}

\newcommand{\extd}{\mathrm{ext}} 
\newcommand{\cl}{\mathrm{cl}} 

\newcommand{\Ga}{\mathbb{G}_a}

\DeclareMathOperator{\GL}{GL}

\newcommand{\lr}{\longrightarrow}

\def\git{/\!\!/}

\newcommand{\oO}{\mathcal{O}}

\newcommand{\mM}{\mathcal{M}}

\newcommand{\rR}{\mathcal{R}}

\renewcommand{\tilde}{\widetilde}
\renewcommand{\hat}{\widehat}
\newcommand{\BL}{{\mathbb{L}}}

\makeatletter
\def\@tocline#1#2#3#4#5#6#7{\relax
	\ifnum #1>\c@tocdepth 
	\else
	\par \addpenalty\@secpenalty\addvspace{#2}%
	\begingroup \hyphenpenalty\@M
	\@ifempty{#4}{%
		\@tempdima\csname r@tocindent\number#1\endcsname\relax
	}{%
		\@tempdima#4\relax
	}%
	\parindent\z@ \leftskip#3\relax \advance\leftskip\@tempdima\relax
	\rightskip\@pnumwidth plus4em \parfillskip-\@pnumwidth
	#5\leavevmode\hskip-\@tempdima
	\ifcase #1
	\or\or \hskip 1em \or \hskip 2em \else \hskip 3em \fi%
	#6\nobreak\relax
	\hfill\hbox to\@pnumwidth{\@tocpagenum{#7}}\par
	\nobreak
	\endgroup
	\fi}
\makeatother

\title[Derived Good Moduli Spaces]{Good Moduli Spaces in Derived Algebraic Geometry}

\author[E. Ahlqvist]{Eric Ahlqvist}
\address{Department of Mathematics, Stockholm University, 106 91 Stockholm, Sweden}
\email{eric.ahlqvist@math.su.se}

\author[J. Hekking]{Jeroen Hekking} 
\address{Department of Mathematics, Stockholm University, 106 91 Stockholm, Sweden}
\email{jeroen.hekking@math.su.se}

\author[M. Pernice]{Michele Pernice}
\address{Department of Mathematics,
University of Washington, 98195 Seattle (WA), USA}
\email{mpernice@uw.edu}

\author[M. Savvas]{Michail Savvas}
\address{Radix Trading, New York, NY, USA}
\email{michail.d.savvas@gmail.com}

\date{\today}

\begin{document}

\begin{abstract}
		We develop a theory of good moduli spaces for derived Artin stacks, which naturally generalizes the classical theory of good moduli spaces introduced by Alper. As such, many of the fundamental results and properties regarding good moduli spaces for classical Artin stacks carry over to the derived context. In fact, under natural assumptions, often satisfied in practice, we show that the derived theory essentially reduces to the classical theory. As applications, we establish derived versions of the \'{e}tale slice theorem for good moduli spaces and the partial desingularization procedure of good moduli spaces.            
	\end{abstract}
 
\thanks{Eric Ahlqvist was supported by the Knut and Alice Wallenberg Foundation 2021.0279 and 2024.0276.}
\thanks{Jeroen Hekking was supported by the Knut and Alice Wallenberg Foundation 2021.0287, and by the Collaborative Research Centre SFB 1085 \emph{Higher Invariants -- Interactions between Arithmetic Geometry and Global Analysis}, project number 224262486.}
\thanks{Michele Pernice was supported by the Knut and Alice Wallenberg Foundation 2021.0291.}

\maketitle

\setcounter{tocdepth}{1}
\tableofcontents
\setcounter{tocdepth}{2}

\section*{Introduction}

\subsection*{Some brief history} Moduli theory focuses on the understanding of geometric objects and how they can vary in families. The fundamental gadget which the modern foundations of the subject rest on is the notion of an algebraic stack, first introduced in \cite{DeligneMumfordSt} and later generalized in \cite{ArtinStacks}. A parallel, important consideration, which is often useful in the study of the properties and behaviour of algebraic stacks, is the construction of a moduli space, i.e., a scheme or algebraic space, which keeps track of sufficient information about the stack, e.g., certain equivalence classes of the objects it parametrizes. 

Mumford developed Geometric Invariant Theory (GIT) \cite{MumfordGIT} in order to be able to define moduli spaces for quotient stacks for reductive group actions on projective schemes equipped with appropriate extrinsic data (a linearization of the action). GIT has led to constructions of moduli spaces of curves and abelian varieties and has been widely applied in a plethora of other contexts, such as moduli spaces of sheaves (cf.~\cite{HuyLehn}). 

Still, many important examples of interest that arise in practice do not fall into the GIT regime. Keel--Mori \cite{KeelMori} first partially lifted this restriction by showing that any stack with finite inertia admits a coarse moduli space. Subsequently, Alper \cite{AlperGood} introduced and developed \cite{AlperFundLemma, AlperLocalProp} the theory of good moduli spaces, which should be viewed as a complete generalization of GIT and allows for the construction of good moduli spaces in a wide variety of new settings. Paired with recent results on existence criteria \cite{AlperExistence}, Alper's theory has been extremely influential and allowed for important breakthroughs in moduli theory, including moduli of objects in abelian categories \cite{AlperExistence} and $K$-semistable Fano varieties in the minimal model program \cite{GMSKStab}.
\smallskip

The purpose of this paper is to initiate the study of good moduli spaces in derived algebraic geometry, generalizing the above theory of good moduli spaces for classical stacks. Derived algebraic stacks \cite{HAG2DAG} can be thought of as (co)homological thickenings of classical stacks and materialize the hidden smoothness principle of Kontsevich \cite{Konts}, with singular classical stacks being truncations of better behaved derived stacks. Even for classical applications, recent breakthroughs in derived algebraic geometry \cite{ToenDerived} and shifted symplectic structures \cite{PTVV} have made derived stacks increasingly important. As such, the existence of derived good moduli spaces for derived stacks commonly used in practice is a desirable feature.

\subsection*{Statement of results} We now proceed to briefly state our main results. For purposes of simplicity in this introduction, we assume that all derived Artin stacks are (homotopically) finitely presented and qcqs over $\CC$ in what follows. We refer the reader to the main text for the precise assumptions made in each statement.

Here is our main definition.

\begin{definition*}[Definition~\ref{def:gms}]
    A morphism $q \colon X \to Y$ from a derived Artin stack to a derived algebraic space $Y$ is a good moduli space if it satisfies the following conditions:
    \begin{enumerate}
        \item $q$ is universally of cohomological dimension $0$.
        \item The natural morphism $\oO_Y \to q_\ast \oO_X$ is an equivalence.
    \end{enumerate}
\end{definition*}

We comment that when $X$ is classical, we show that $Y$ must be classical as well (see Proposition~\ref{Prop:Classical stack has classical GMS}). In fact, one of our main results says that the existence of a good moduli space can be detected at the level of the classical underlying stack.

\begin{theorem*} [Theorem~\ref{Thm:GMSvsGMScl}]
(i) If $X_\cl$ admits a good moduli space $q_\cl \colon X_\cl \to Y_\cl$, then $X$ admits a good moduli space $q' \colon X \to Y$ such that $q'_\cl \simeq q_\cl$.

(ii) If $X$ admits a good moduli space $q\colon X \to Y$, then $q_\cl$ is a good moduli space for $X_\cl$.
\end{theorem*}

Interestingly, while condition (ii) of our definition is identical to the $\oO$-connectedness condition imposed by Alper, condition (i) could in general be stronger than requiring that $q$ is cohomologically affine, as Alper does. The reason we need to impose a stronger condition is to ensure that the pushforward 
functor $q_\ast$ is sufficiently well-behaved on derived stacks. In practice, this distinction is fairly benign, as it disappears for Noetherian $X$ with quasi-affine diagonal (see Proposition~\ref{Prop:comparison between Alper and derived GMS}), and so our moduli spaces are the same as the ones defined by Alper in that case. See Subsection~\ref{subsection:different notions of morphisms and gms} for a discussion of different notions of good moduli spaces.

Another main result of the paper is the universality of good moduli spaces.

\begin{theorem*}[Theorem~\ref{thm:universality of gms}]
Good moduli spaces are universal for maps to derived algebraic spaces.
\end{theorem*}

Following Alper's steps \cite{AlperGood}, we establish several properties of good moduli spaces, which mirror the classical theory and are listed in Lemma~\ref{Lem:GMS_base-change} and Lemma~\ref{lem:gms_prop}. Additionally, we prove results on good moduli spaces of closed substacks, gluing good moduli spaces and the descent of \'{e}tale morphisms from stacks to their good moduli spaces.
\smallskip

We provide several applications of our theory. The first is a fully derived version of the \'{e}tale slice theorem for good moduli spaces \cite{AlperLuna}. A simplified statement sketch is below (see the main text for all assumptions).

\begin{theorem*}[Theorem~\ref{thm:derived slice thm}]
Suppose that $X$ is a derived Artin stack with good moduli space $q\colon X\to Y$ with affine diagonal. Let $ x \in |X_\cl|$ be a closed point with (reductive) stabilizer $G_x$. 

Then, there exists a strongly \'{e}tale morphism $\Phi \colon [U / G_x] \to X$, where $U$ is a derived affine scheme with a $G_x$-action, $u \in U$ is fixed by $G_x$ and maps to $x$ via $\Phi$, fitting in a Cartesian square
    \begin{align*}
        \xymatrix{
        [U / G_x] \ar[d]_-{q'} \ar[r]^-{\Phi} & X \ar[d]^-{q} \\
        U \git G_x \ar[r] & Y,
        }
    \end{align*}
where $U \git G_x$ denotes the good moduli space of $[U/G_x]$.
\end{theorem*}

A second application concerns the stabilizer reduction procedure for derived Artin stacks. Given an Artin stack $X$ with affine diagonal such that $X_\cl$ admits a good moduli space, in \cite{HRS}, the authors defined a canonical derived Deligne--Mumford stack $\tilde{X} \to X$, which resolves the positive-dimensional stabilizers of $X$ via an iterated blow-up construction (see also \cite{KiemLi,KLS,Sav} for classical versions of this process). While it was shown that $\tilde{X}_\cl$ also admits a good moduli space, we now prove that the same holds for $\tilde{X}$ itself, as expected. More precisely, the good moduli space of $\tilde{X}$ should be viewed as a derived version of Kirwan's partial desingularization of GIT quotients \cite{Kirwan} and Edidin--Rydh's stabilizer reduction procedure \cite{EdidinRydh}.

\begin{theorem*}[Theorem~\ref{thm:gms of kirwan blow-up}, Corollary~\ref{corr:gms of stab red}]
Suppose that $X$ is a derived Artin stack with affine diagonal and a good moduli space $q \colon X \to Y$. Then the same holds for its stabilizer reduction $\tilde{X}$.
\end{theorem*}

Finally, we also construct natural derived enhancements of fundamental examples of classical good moduli spaces, such as moduli spaces of stable maps and sheaves, and give a concrete description of the good moduli space of affine quotient stacks $[\Spec A / G]$ for a linearly reductive group $G$.

\subsection*{Future directions} A natural extension of the present paper is to prove an \'{e}tale slice theorem for good moduli spaces (see~\cite[Theorem~6.1]{AHR2} and \cite[Theorem~4.12]{AlperLuna}) and a partial desingularization result for stacks over \emph{any} base scheme $S$ and in particular over mixed characteristic. Another extension is to prove coherent completeness for Noetherian derived Artin stacks admitting a good moduli space \cite[Thm.~1.6]{AHR2}.

A natural direction for future inquiry is to obtain a theory of derived adequate moduli spaces \cite{AlperAdequate}, which will cover the case of Artin stacks over characteristic $p>0$, where the notion of a linearly reductive group is more restrictive and differs from being reductive. Giving the correct condition, replacing universally of cohomological dimension $0$, looks like a subtler task, which we hope will be within reach in the near future.

It will also be interesting to exhibit explicit, interesting, non-trivial examples of derived good moduli spaces which enjoy nice properties, such as being quasi-smooth, and develop the theory further to apply to higher stacks.

Finally, we intend to establish a general derived Luna's fundamental lemma for the descent of properties of morphisms (including and going beyond \'{e}taleness) of algebraic stacks to their good moduli spaces generalizing \cite{LunaRydh}.

\subsection*{Layout of the paper} In \S\ref{sec:2}, we gather the necessary results and properties regarding morphisms which are universally of cohomological dimension $0$ and related notions. In \S\ref{sec:3}, we define good moduli spaces and prove our main results and several properties. \S\ref{sec:4} is concerned with applications and examples, including an \'{e}tale slice theorem, stabilizer reduction, an existence criterion and quotient stacks. Finally, we recall the basics of projective spectra and blow-ups in derived geometry in two appendices.

\subsection*{Acknowledgements} MS would like to thank Jarod Alper, Daniel Halpern-Leistner and David Rydh for discussions related to good moduli spaces and the \'{e}tale slice theorem. EA and JH would like to thank David Rydh for inspiring conversations on good moduli spaces and their applications. EA would also like to thank Ben Davison and Sebastian Schlegel Mejia for useful conversations. Finally, we are collectively indebted to David Rydh for carefully going through an earlier draft and providing numerous comments, corrections and valuable suggestions that resulted in an overall great improvement of this paper.

\subsection*{Notation and conventions}
Throughout, everything is derived and $\infty$-categorical. Hence a stack is a space-valued sheaf $\Aff^\op \to \Space$ on the category of (derived) affine schemes. While the results of this paper preceding Section~\ref{sec:4} hold over a base derived algebraic space $S$, we typically work over $\Spec \ZZ$ for simplicity, as this does not affect the essence of our results. Indeed, the universality of good moduli spaces for maps to derived algebraic spaces implies that the relative case reduces to the absolute case.

An \emph{Artin stack} is a 1-algebraic stack $X$. If moreover $X_\cl$ is a classical algebraic space, then $X$ is an \emph{algebraic space}.

For any stack $X$, the \emph{category of quasi-coherent modules} $\QCoh(X)$ is the limit of the categories $\Mod_A$, indexed over all $\Spec A \to X$. The category $\QCoh(X)$ is endowed with a canonical $t$-structure, where $M \in \QCoh(X)$ is connective if and only if $M_B \coloneqq x^*M$ is connective for all $x \colon \Spec B \to X$.

A morphism $f \colon X \to Y$ of Artin stacks is called \emph{qcqs} if $f_\cl$ is quasi-compact and quasi-separated. 

Recall that an exact functor between stable categories is \emph{right} $t$-exact if it preserves connective objects, and that the dual notion is \emph{left} $t$-exact. We use homological indexing notation unless otherwise stated.

In a stable category, (co)fiber sequences are also called \emph{exact} sequences.

A morphism of algebras, modules, ..., is \emph{surjective} if the fiber is connective. 

For an Artin stack $X$, we write $\lvert X \rvert \coloneqq \lvert X_\cl \rvert$ for the topological space of points of $X$.

We write $\Map(-,-)$ for mapping spaces.

\section{Connectivity and pushforward} \label{sec:2}

The purpose of this section is to gather and develop the requisite background and results for a well-behaved pushforward functor in the context of derived algebraic geometry. This will be necessary towards a robust definition of good moduli spaces later on, which will rely on the properties of pushing forward. The appropriate notion of morphisms that we will consider is being universally of cohomological dimension zero, following \cite{HalpernleistnerMapping}. We will see that this guarantees the desired properties and bears close relation to affineness and $t$-exactness.
\medskip

Let $f \colon X \to Y$ be a morphism of Artin stacks. Recall that we have an adjunction
\[ (f^* \dashv f_*) \colon \QCoh(Y) \rightleftarrows \QCoh(X) \]
such that $f^*$ is right $t$-exact and $f_*$ is left $t$-exact.

For $y \colon \Spec A \to Y$, we write $f_A \colon X_A \to \Spec A$ for the pullback of $f$ along $y$.

\subsection{Different notions of affineness} We briefly discuss different notions of affineness and their relations.
\begin{remark}
    Let $f \colon X \to Y$ be a morphism of Artin stacks. Consider the functors
    \begin{align*}
        f_*^{\heartsuit} &\colon \QCoh(X)^{\heartsuit} \xrightarrow{f_*} \QCoh(Y)_{\leq 0} \xrightarrow{\tau_{\geq 0}} \QCoh(Y)^{\heartsuit} \\
        f^*_\heartsuit & \colon \QCoh(Y)^{\heartsuit} \xrightarrow{f^*} \QCoh(X)_{\geq 0} \xrightarrow{\tau_{\leq 0}} \QCoh(X)^{\heartsuit}.
    \end{align*}
    Using the adjoint properties of $\tau_{\geq 0}$ and $\tau_{\leq 0}$, one shows that $f^*_\heartsuit \dashv f_*^{\heartsuit}$. If $X,Y$ are classical and $f$ is qcqs, then these functors agree with the underived pushforward and pullback functors \cite[Lem.~1.2]{HallRydhPerfect}.
\end{remark}

\begin{definition}
    A morphism $f \colon X \to Y$ of Artin stacks is called
    \begin{itemize}
        \item \textit{cohomologically affine} if $f_*^{\heartsuit} \colon \QCoh(X)^\heartsuit \to \QCoh(Y)^\heartsuit$ is exact,
        \item \textit{universally cohomologically affine} if $(f')_*^{\heartsuit} \colon \QCoh(X')^\heartsuit \to \QCoh(Y')^\heartsuit$ is exact for any Cartesian square of Artin stacks
        \begin{align*}
            \xymatrix{
            X' \ar[d]^-{f'} \ar[r] & X \ar[d]^-{f} \\
            Y' \ar[r] & Y,
            }
        \end{align*}
        \item \textit{affine} if for all morphisms $y \colon \Spec A \to Y$, the pullback $X_A$ is an affine scheme, and
        \item \emph{quasi-affine} if it factorizes as a quasi-compact open immersion $X \to \overline{X}$ followed by an affine morphism $\overline{X} \to Y$. 
    \end{itemize}
    If $X \to *$ is cohomologically affine (quasi-affine), then $X$ is called \emph{cohomologically affine} (\emph{quasi-affine}).
\end{definition}

It is clear that affineness is stronger than cohomological affineness and quasi-affineness. 

\begin{remark}
\label{Rem:coh_aff_cl}
    Recall that pushing forward along $Z_\cl \to Z$ induces an equivalence $\QCoh(Z_\cl)^\heartsuit \simeq \QCoh(Z)^{\heartsuit}$, for any Artin stack $Z$. It follows that $f \colon X \to Y$ is cohomologically affine if and only if $f_\cl$ is.
\end{remark}

\begin{remark}
\label{Rem:noncon_push}
Recall that a qcqs morphism $f \colon X \to Y$ of classical algebraic stacks is quasi-affine if and only if the canonical map $X \to \Spec_Y(f^{\heartsuit}_*\oO_X)$ is a quasi-compact open immersion. To get a similar result in the derived setting, we need to deal with the fact that $f_*\oO_X$ might be nonconnective. This can be done via nonconnective derived geometry (as developed in \cite{BenBassatHekkingBlowups}, based on the theory of derived rings by Bhatt--Mathew \cite{RaksitHKR}). We sketch the argument.

Write $\Alg^{\nc}$ for the category of nonconnective algebras (i.e., derived rings in $\Mod_\ZZ$). Then the inclusion $\iota \colon \Alg \to \Alg^{\nc}$ from algebras into nonconnective algebras has a right adjoint $\tau_{\geq 0}(-)$, which is the truncation on underlying modules. Put $\Aff \coloneqq \Alg^{\op}$ and $\Aff^\nc \coloneqq (\Alg^{\nc})^{\op}$.

The inclusion $i \colon \Stk \to \Stk^{\nc}$ from stacks to nonconnective stacks has a left adjoint $t \colon \Stk^{\nc} \to \Stk$, such that $t \dashv i$ restricts to the adjunction
\[ (t \dashv i) \colon \Aff^\nc \rightleftarrows \Aff  \]
induced by $\iota \dashv \tau_{\geq 0}$.

Yoneda induces a contravariant inclusion $\Spec^\nc(-) \colon \Alg^\nc \to \Stk^\nc$. Define $\QAlg^\nc(-) \colon \Stk^\nc \to \Cat$ by right Kan extension of $\Alg^\nc_{(-)} \colon \Aff^\nc \to \Cat$. For $S \in \Stk$, we get an inclusion $\Spec^\nc_S(-) \colon \QAlg^\nc(S) \to \Stk^\nc_{/S}$. For $T \in \Stk$, write $\QAlg(iT)$ for the full subcategory of $\QAlg^\nc(iT)$ spanned by connective $\oO_{iT}$-algebras. Then $\QAlg(i(-))$ is also the right Kan extension of $\Alg_{(-)} \colon \Aff \to \Cat$, hence coincides with the existing definition in the connective setting.

Let $f \colon X \to Y$ be a morphism of stacks. The truncation functor induces an adjunction
\[ (\iota \dashv \tau_{\geq 0}) \colon \QAlg(iY) \rightleftarrows \QAlg^\nc(iY) \]
which commutes with pullbacks to affine schemes. It follows that the functor $f^* \colon \QAlg(Y) \to \QAlg(X)$ has a right adjoint
\[ \QAlg(iX) \xrightarrow{\iota} \QAlg^\nc(iX) \xrightarrow{(if)_*} \QAlg^\nc(iY) \xrightarrow{\tau_{\geq 0}(-)} \QAlg(iY) \]
after identifying $\QAlg(iX) \simeq \QAlg(X)$ and $\QAlg(iY) \simeq \QAlg(Y)$. We write this adjoint as $\tau_{\geq 0}f_*(-)$.

One can now show that $\Spec_Y(\tau_{\geq 0}f_*B) \simeq t \Spec^\nc_{iY}((if)_*B)$, by reducing to the case where $Y$ is affine, and then using that $t$ restricts to the functor $\Aff^\nc \to \Aff$ induced by $\tau_{\geq 0}$. Hence, upon applying $t$ to the factorization
\[ iX \to \Spec^\nc_{iY}((if)_*\oO_X) \to iY \]
of $if$, we obtain the factorization
\[ X \to \Spec_Y(\tau_{\geq 0}f_*\oO_X) \to Y \]
of $f$.
\end{remark}

\begin{lemma}
\label{Lem:qaffine_bc}
    Let $f \colon X \to Y$ be a morphism of Artin stacks. Then the following are equivalent.
    \begin{enumerate}
        \item $f$ is quasi-affine.
        \item $X_A$ is quasi-affine, for all $\Spec A \to Y$.
        \item $X \to \Spec_Y \tau_{\geq 0} f_* \oO_X$ is a quasi-compact open immersion.
        \item $f_\cl$ is quasi-affine.
    \end{enumerate}
\end{lemma}

\begin{proof}
    That (i) implies (ii)  is clear since open immersions and affine morphisms are stable under base-change. Also (iii) implies (i) by definition.
    
    Now assume (ii). Then $f$ is quasi-compact schematic, so the construction $X \to \Spec_Y \tau_{\geq 0} f_* \oO_X$ is stable under base-change on $Y$ by \cite[Ch.~2, Prop.~2.2.2]{GaitsgoryStudy}. To show (iii), we may thus assume that $Y$ is affine, hence that $X$ is quasi-affine. Take a quasi-compact open immersion $g \colon X \to \Spec R$, and put $B \coloneqq \tau_{\geq 0}g_*\oO_X$. By Remark \ref{Rem:noncon_push}, we obtain a factorization
    \[ X \to \Spec B \to \Spec R \]
    of $f$. Let $U \to \Spec R$ be the open subscheme such that $\lvert U \rvert$ is the image of $\lvert X \rvert \to \lvert \Spec R \rvert$. Since $X$ is quasi-compact, $U$ is covered by a finite number of schemes of the form $\Spec R_s$, with $s \in R$. For any such $s$, consider the Cartesian diagram
    \begin{center}
        \begin{tikzcd}
            X_s \arrow[r] \arrow[d, "h'"] \arrow[dd, bend right=60, swap, "g'"] & X \arrow[d, "h"] \arrow[dd, bend left=60, "g"] \\
            \Spec j^*B \arrow[r] \arrow[d] & \Spec B \arrow[d] \\
            \Spec R_s \arrow[r, "j"] & \Spec R
        \end{tikzcd}
    \end{center}
    Since $f$ is an open immersion and $j$ lands in the image of $g$, the morphism $g'$ is equivalent to $\Spec R_s \to \Spec B_s \to \Spec R_s$. On the other hand, again by \cite[Ch.~2, Prop.~2.2.2]{GaitsgoryStudy} and by Remark \ref{Rem:noncon_push}, we have
    \[ j^*B \simeq \tau_{\geq 0} j^*g_* \oO_X \simeq \tau_{\geq 0} f'_*\oO_{X_s} \simeq R_s. \]
    It follows that $h'$ is an equivalence, hence that $h$ is a quasi-compact open immersion. 
    
    Clearly (i) implies (iv) so it remains to show that (iv) implies (i). By the preceding argument we again reduce to the case where $X$ is a quasi-affine scheme. Assume $f_\cl$ is quasi-affine, and put $B \coloneqq \tau_{\geq 0}\Gamma(X,\oO_X)$. Let $X'$ be the open subscheme of $\Spec B$ determined by the open map $\lvert X \rvert \to \lvert \Spec B \rvert$. Since $X$ is quasi-compact, $X'$ is covered by a finite number of schemes of the form $\Spec B_s$. For any such $s$, we obtain a Cartesian diagram
    \begin{center}
        \begin{tikzcd}
            X_s \arrow[d, "\alpha"] \arrow[r] & X \arrow[d] \\
            X'_s \arrow[d, "\beta"] \arrow[r] & X' \arrow[d] \\
            \Spec B_s \arrow[r] & \Spec B
        \end{tikzcd}
    \end{center}
    such that $\beta$ is an equivalence. With a similar argument as before, it holds that $B_s \simeq \tau_{\geq 0}\Gamma(X_s,\oO_{X_s})$. It thus suffices to show that $X_s$ is affine. Since $(X_s)_\cl$ is affine, this follows from \cite[Cor.~1.1.6.3]{LurieSpectral} (or from topological invariance).
\end{proof}

Let $f \colon X \to Y$ be a morphism of Artin stacks. Then $f$ is \emph{locally quasi-finite} if $f_\cl$ is locally quasi-finite. By \cite[\href{https://stacks.math.columbia.edu/tag/0397}{Tag 0397}]{stacks-project}, this definition coincides with the one given in \cite[Def.~3.3.1.1]{LurieSpectral} in case $X,Y$ are algebraic spaces. Recall that $f$ is \emph{separated} if $f_\cl$ is.

\begin{corollary}
\label{Cor:Zar_main_thm}
    A representable morphism of Artin stacks which is quasi-compact, separated and locally quasi-finite is quasi-affine.
\end{corollary}

\begin{proof}
    By Lemma \ref{Lem:qaffine_bc} this reduces to the case of classical algebraic spaces, which follows from Zariski's main theorem \cite[\href{https://stacks.math.columbia.edu/tag/082J}{Tag 082J}]{stacks-project}.
\end{proof}

\begin{example}
    \label{Ex:algsp_quasiaff}
    The diagonal $\Delta$ of a quasi-separated algebraic space is quasi-affine. Indeed, by Corollary \ref{Cor:Zar_main_thm}, this follows from \cite[\href{https://stacks.math.columbia.edu/tag/02X4}{Tag 02X4}]{stacks-project}.     
\end{example}

\subsection{Morphisms universally of cohomological dimension zero} We now introduce the central definition of this section.

\begin{definition}
    A morphism $f \colon X \to Y$ of Artin stacks is \emph{of cohomological dimension zero} (\cd for short) if for all $M \in \QCoh(X)^{\heartsuit}$ it holds that $f_\ast M$ is connective (and hence discrete).
    
    It is \emph{universally of cohomological dimension zero} (\ucd for short) if for all morphisms $\Spec A \to Y$ and all $M \in \QCoh(X_A)^{\heartsuit}$ it holds that $(f_A)_*M$ is connective.
\end{definition}

\begin{remark}
\label{Rem:coh_dim_cl}
    By the same reasoning as in Remark \ref{Rem:coh_aff_cl}, it holds that $f$ is \cd if and only if $f_\cl$ is. Hence the same holds true for being \ucd.
\end{remark}

Recall that a morphism $f \colon X \to Y$ of stacks  satisfies the \emph{base-change formula} if for every Cartesian square
of stacks
	\begin{equation}
		\label{E:base-change-square}
		\begin{tikzcd}
			X' \arrow[d, "f'"] \arrow[r, "g'"] & X \arrow[d, "f"] \\
			Y' \arrow[r, "g"] & Y, 
		\end{tikzcd}
	\end{equation}
 the natural transformation $g^*f_* \to f_*' (g')^*$ is an equivalence. Related to this, we say that $f$ satisfies the \emph{projection formula} if, for all $M \in \QCoh(X)$ and $N \in \QCoh(Y)$, the natural map $(f_*M) \otimes_{\oO_Y} N \to f_*(M\otimes_{\oO_X} f^*N)$ is an equivalence.

\begin{lemma}
\label{Lem:ucd0_props}
    Being \ucd is stable under composition and base-change. If $f$ is qcqs and \ucd, then $f_*$ preserves filtered colimits and satisfies the base-change and projection formulas.
\end{lemma}

\begin{proof}
    This is \cite[Prop.~A.1.5, Lem.~A.1.6]{HalpernleistnerMapping}.
\end{proof} 

\begin{remark} \label{rem:2.4}
    Let $f \colon X \to Y$ be a morphism of Artin stacks. Since $f_*$ is always left $t$-exact, it follows that $f_*$ is $t$-exact if and only if $f_*(\QCoh(X)_{\geq 0}) \subseteq \QCoh(Y)_{\geq 0}$.

    The proof of \cite[Lem.~A.1.6]{HalpernleistnerMapping} assumes qcqs, but only uses that the categories of quasi-coherent modules in question are presentable and $t$-complete, which is true for all Artin stacks \cite[Ch.~3, Cor.~1.5.7]{GaitsgoryStudy}. It follows that $f$ is \cd if and only if $f_*$ is $t$-exact.
\end{remark}

We have the following cancellation property (cf.~\cite[Prop.~3.13]{AlperGood}), which will be of use later on.

\begin{proposition} \label{Prop:cancellation property}
    Let $f \colon X \to Y$ and $g \colon Y \to Z$ be morphisms of Artin stacks. Suppose that $g \circ f$ is \ucd and $g$ has \ucd diagonal. Then $f$ is \ucd.
\end{proposition}

\begin{proof}
    Consider the diagram with Cartesian squares
    \begin{align*}
    \xymatrix{
    & X \ar[r]^-{(\id, f)} \ar[dl] & X \times_Z Y \ar[r] \ar[dr] \ar[dl] & Y \ar[dr] \\
    Y \ar[r]_-{\Delta} & Y \times_Z Y & & X \ar[r] & Z.
    }
    \end{align*}
    By assumption on $\Delta$ and the base-change property of being \ucd, it follows that $f$ is the composition of two \ucd morphisms and hence \ucd.
\end{proof}

The following proposition shows that being \ucd is equivalent to being universally $t$-exact.

\begin{proposition}
\label{Prop:ucd0_equivalent}
    Let $f \colon X \to Y$ be a morphism of Artin stacks. Then the following are equivalent:
    \begin{enumerate}
        \item $f$ is \ucd.
        \item For all $Y' \to Y$ with $Y'$ Artin, it holds that $f'_*$ is $t$-exact, where $f' \colon X' \to Y'$ is the pullback of $f$ along $Y' \to Y$.
        \item For all $\Spec A \to Y$ it holds that $(f_A)_*$ is $t$-exact.
    \end{enumerate}
\end{proposition}

\begin{proof}
    Clearly we have that (ii) implies (iii) and that (iii) implies (i). By Lemma \ref{Lem:ucd0_props} and Remark \ref{rem:2.4}, (i) implies (ii).   
\end{proof}

When the target of a morphism has quasi-affine diagonal, universality is no longer a condition, as shown in the next proposition.

\begin{proposition} \label{Prop:quasi-aff_cd_ucd_$t$-exact}
    Let $f \colon X \to Y$ be a qcqs morphism of Artin stacks and assume that $Y$ has quasi-affine diagonal. Then the following are equivalent:
    \begin{enumerate}
        \item $f$ is \ucd.
        \item $f$ is \cd.
        \item $f_*$ is $t$-exact.
    \end{enumerate}
\end{proposition}

\begin{proof}
    By definition, (i) implies (ii) and Remark~\ref{rem:2.4} shows that (ii) is equivalent to (iii). It suffices to show that (ii) implies (i).
    
    Let $g \colon \Spec A \to Y$ be a morphism. Since $Y$ has quasi-affine diagonal, $g$ is quasi-affine by Lemma \ref{Lem:qaffine_bc}. Now, as in the proof of \cite[Prop.~A.1.9]{HalpernleistnerMapping}, being \cd is preserved under affine base-change. We thus need to show that it is also preserved by base-change by open immersions, hence may assume that $g$ is an open immersion. Consider the Cartesian diagram
    \begin{align*}
        \xymatrix{
        X_A \ar[d]_-{f_A} \ar[r]^-{g_A} & X \ar[d]^-f \\
        \Spec A \ar[r]_-{g} & Y.
        }
    \end{align*}
    
    Let $M \in \QCoh(X_A)^{\heartsuit}$. Since $(g_A)_*$ is monic, the counit $g_A^*({g_A})_* M \to M$ is an equivalence, hence, using the fact that $g_A^\ast$ is $t$-exact, so is
    \[ g_A^*\pi_0 ( {g_A}_* M ) \simeq \pi_0 (g_A^*{g_A}_* M ) \to \pi_0 M \simeq M. \]
    We may thus assume that $M$ is of the form $(g_A)^\ast N$ for some $N \in \QCoh(X)^{\heartsuit}$. By \cite[Lemma~A.1.3]{HalpernleistnerMapping}, since $g$ is flat and $M \in \QCoh(X_A)^{\heartsuit}$, we get an equivalence 
    $$ (f_A)_\ast M \simeq (f_A)_\ast (g_A)^\ast N \simeq g^\ast f_\ast N \in \QCoh(\Spec A)^{\heartsuit} $$
    since $f$ is \cd and $g^\ast$ is $t$-exact. Thus $f_A$ is \cd, as we want.
\end{proof}

\begin{remark}
    The implication from (ii) to (i) in the above proof also follows from \cite[Lemma~2.2(6)]{HallRydhPerfect}.
\end{remark}

\begin{remark}
    If $Y$ does not have affine stabilizers, then being \cd does not imply being \ucd, see \cite[Rem.~1.6]{HallRydhGroups} for an example.
\end{remark}

The following proposition establishes a connection with cohomological affineness in the classical case, as defined in \cite{AlperGood}. Before giving the statement, recall that an Artin stack $X$ is \emph{affine-pointed} if every morphism $\Spec k \to X$, where $k$ is a field, is affine. If $X$ has quasi-affine diagonal then it is affine-pointed \cite[Lemma~4.6]{HRAffPointed}. Part (iv) owes its existence to discussions with David Rydh.

\begin{proposition}
\label{Prop:ucd0_props}
    Let $f \colon X \to Y$ be a qcqs morphism of Artin stacks.
    \begin{enumerate}
        \item If $f$ is \cd, then it is cohomologically affine. If $f$ is \ucd, then it is universally cohomologically affine.
        \item Suppose that $X$ and $Y$ either are both affine-pointed with $X_\cl$ Noetherian or both have affine diagonal with $X$ quasi-compact. If $f$ is cohomologically affine, then it is \cd. 
        \item Suppose that $X$ and $Y$ have quasi-affine diagonal with $X_\cl$ Noetherian. If $f$ is cohomologically affine, then it is \ucd.
        \item Suppose that $f$ has affine diagonal. Then $f$ is \ucd if and only if it is universally cohomologically affine.
    \end{enumerate}
\end{proposition}

\begin{proof}
    By Remark \ref{Rem:coh_aff_cl} and Remark \ref{Rem:coh_dim_cl}, we may assume that $X,Y$ are classical.

    If $f$ is \cd, then $f_*$ restricts to a functor $\QCoh(X)^{\heartsuit} \to \QCoh(Y)_{\geq 0}$. Since $f_*$ is always left $t$-exact, $f_*$ then in fact restricts to $\QCoh(X)^{\heartsuit} \to \QCoh(Y)^{\heartsuit}$, which is thus exact. The same argument applies after base change to imply the universal case. This shows (i).
    
    Recall that $\QCoh(X)^\heartsuit$ has enough injectives \cite[\href{https://stacks.math.columbia.edu/tag/0781}{Tag 0781}]{stacks-project}. Part (ii) is a direct consequence of \cite[Lem.~C.3]{HallNeemanRydh}, which implies that for any $M \in \QCoh(X)^{\heartsuit}$ and an injective resolution $M \to I^\bullet \in \QCoh(X)^{\heartsuit}$, there is a natural equivalence between $f_\ast M$ and $f_\ast^{\heartsuit} I^\bullet$. Thus, exactness of $f_\ast^{\heartsuit}$ implies that $f_\ast^{\heartsuit} I^\bullet \in \QCoh(Y)^{\heartsuit}$ and hence $f_\ast M \in \QCoh(Y)^{\heartsuit}$, so that $f$ is \cd.

    Part (iii) follows from part (ii) and Proposition~\ref{Prop:quasi-aff_cd_ucd_$t$-exact}, since having quasi-affine diagonal implies being affine-pointed. 

    Finally, for part (iv), the general fact that \ucd implies universal cohomological affineness follows from part (i). For the converse, if $\Spec A \to Y$ is any morphism, then the pullback $f_A \colon X_A \to \Spec A$ is qcqs and has affine diagonal, so since $\Spec A \to \ast$ has affine diagonal, the same holds for $X_A$. Moreover, $X_A$ is quasi-compact, since the same is true for $\Spec A$. By part (ii), we deduce that $f_A$ is \cd, and hence $f$ is \ucd, as desired.
    \end{proof}
    
    One may wonder if cohomologically affine (respectively universally cohomologically affine) implies \cd (respectively \ucd), in general. The problem lies in the fact that if the conditions in (ii) are not met, it is not necessarily true that the derived functor $Rf_*^{\heartsuit}$ coincides with $f_*$, as the following example shows.

\begin{example}[Hall--Rydh] \label{example:gms1}
    Let $A$ be a classical abelian variety of dimension $\geq 1$ over $\ast = \Spec \CC$, e.g., an elliptic curve, and $f$ the universal torsor $\ast \to BA$. Then $f$ is cohomologically affine, but neither \cd nor universally cohomologically affine.

    First, recall that $A$ is proper and connected over $\CC$, hence $\Gamma(A,\oO_A)^\heartsuit \simeq \CC$  \cite[\href{https://stacks.math.columbia.edu/tag/0FD2}{Tag 0FD2}]{stacks-project}. 
    
    Second, recall that $\QCoh(BA)^\heartsuit$ is equivalent to the category of locally finite representations of $A$ \cite{HallRydhGroups}. In particular, any $M \in \QCoh(BA)^\heartsuit$ can be written as a colimit $M = \colim M_\alpha$ of finite-dimensional $A$-representations $M_\alpha$. Such $M_\alpha$ is classified by a map $A \to \GL(M_\alpha)$, which is trivial since $\Gamma(A,\oO_A)^\heartsuit \simeq \CC$.
    It follows that $f_*^\heartsuit$ is invertible, hence that $f$ is cohomologically affine. 

    Let $p \colon A \to *$ be the structure map. Since $f$ is schematic and quasi-compact, it holds that
    \[ f^*f_*(\CC) \simeq p_*p^*(\CC) \simeq p_*\oO_A \simeq \Gamma(A,\oO_A) \]
    by \cite[Prop.~2.2.2]{GaitsgoryStudy}. Now $\Gamma(A,\oO_A)$ is non-discrete (for example, by Serre duality), hence so is $f_*(\CC)$ by \cite[Prop.~1.5.4]{GaitsgoryStudy}.

    Finally, observe that $f$ is not universally cohomologically affine since $A \to *$ is not cohomologically affine; for schemes, the latter is equivalent to being affine.
\end{example}

\begin{example}[Rydh] \label{example:gms2}
    For $A,\ast$ as in the previous example, consider the structure map $g \colon BA \to \ast$. Then $g$ is universally cohomologically affine essentially by the same argument as in the previous example.

    But $g$ is not \cd. Indeed, suppose it were. Then $g_*$ would be $t$-exact by Proposition~\ref{Prop:quasi-aff_cd_ucd_$t$-exact}. In particular, we would get
    \[ \pi_n(\CC) \simeq \pi_ng_*f_*(\CC) \simeq g_* \pi_nf_*(\CC) \]
    for all $n$. Since there is some $n \not=0$ for which $\pi_nf_*(\CC) \not=0$ by the previous example, this is absurd (using that $g_*^\heartsuit \simeq \id$).    
\end{example}

\begin{example}\cite[Rem.~1.6]{HallRydhGroups} \label{example:gms3}
    Let $E(A)$ be the universal vector extension of an abelian variety $A$. Then the morphism $\ast \to BE(A)$ is \cd but not universally cohomologically affine.
\end{example}

\begin{example}\cite[Rem.~1.6]{HallRydhGroups} \label{example:gms4}
    Let $E(A)$ be the universal vector extension of an abelian variety $A$. Then the morphism $BE(A) \to \ast$ is \ucd.
\end{example}

Finally, it is easy to see that affine qcqs morphisms are \ucd. The following proposition gives a partial converse. We thank David Rydh for pointing this out to us.

\begin{proposition} \label{prop:rep ucd qcqs implies affine}
    A morphism $f \colon X \to Y$ of Artin stacks is affine if and only if it is representable, qcqs and \ucd.
\end{proposition}

\begin{proof}
    We only need to show that representable, qcqs and \ucd implies affine. Let $\Spec A \to Y$ be a morphism and consider the pullback $f_A \colon X_A \to \Spec A$. Since $f$ is representable, $X_A$ is an algebraic space. By assumption, the morphism $f_A \colon X_A \to \Spec A$ is \ucd, and hence, by Proposition~\ref{Prop:ucd0_props}(i), cohomologically affine. By Serre's criterion \cite[Prop.~3.3]{AlperGood}, $f_A$ is affine, so we conclude that $X_A$ is affine and therefore $f$ is affine.
\end{proof}

\begin{example}
    Any quasi-affine morphism that is not affine is not \ucd. The reason for this is that representable and \ucd morphisms have to be affine by the previous proposition. For example, the quasi-affine open embedding $\AA^2 \setminus \lbrace 0 \rbrace \subseteq \AA^2$ is not \ucd.
\end{example}

\section{Derived good moduli spaces} \label{sec:3}

In this section, we give the main definition of this paper, that of a good moduli space for a (derived) Artin stack. We also establish the main properties of good moduli spaces. Our discussion and results closely follow Alper's original paper \cite{AlperGood} which introduced good moduli spaces for classical Artin stacks. Indeed, our main aim is to provide the appropriate generalizations of these to the setting of derived algebraic geometry. 

We show that, under certain natural assumptions, our definition is consistent with Alper's definition. Moreover, the vast majority of the classical results on good moduli spaces carry over to the derived world, as the property of having a good moduli space can be detected at the level of the classical underlying stack. We also establish fundamental properties, shared with the classical case, such as universality of good moduli spaces for maps to algebraic spaces, among others.

\subsection{Definitions: derived and classical}

\begin{definition} \label{def:gms}
A morphism $q \colon X \to Y$ of Artin stacks is a \emph{good moduli space morphism} if it is qcqs and the following two conditions are satisfied:
\begin{enumerate}
    \item[(i)] $q$ is universally of cohomological dimension zero.
    \item[(ii)] The natural morphism $\oO_Y \to q_\ast \oO_X$ is an equivalence.
\end{enumerate}
If moreover $Y$ is an algebraic space, then $q \colon X \to Y$ is called a \emph{good moduli space} for $X$.
\end{definition}

\begin{remark}
    Let $q : X \to Y$ be a morphism from an Artin stack to a quasi-separated algebraic space.
    We can also equivalently require---using Example~\ref{Ex:algsp_quasiaff} and Proposition~\ref{Prop:quasi-aff_cd_ucd_$t$-exact}---that $q$ is \cd or that $q_*$ is $t$-exact, in order for $Y$ to be a good moduli space for $X$.
\end{remark}

Before we explore properties of good moduli spaces, let us recall the classical definition given by Alper \cite{AlperGood} and compare it with the above definition for a classical Artin stack. This will 
 often allow us to bootstrap from the classical case to the case of derived Artin stacks in what follows. We elect to use the term ``$\sA$-good'' to eliminate possible confusion.
 
\begin{definition}
    An \emph{$\sA$-good moduli space} (resp. \emph{$\mathcal{U}\sA$-good moduli space}) for a classical Artin stack $X$ is a qcqs morphism $q \colon X \to Y$ to a classical algebraic space $Y$ such that $q$ is cohomologically affine (resp. universally cohomologically affine) and $\oO_Y \simeq q_*^{\heartsuit} \oO_X$ under the natural morphism. 
\end{definition}

\begin{remark}
    The stronger notion of $\mathcal{U}\sA$-good moduli spaces has often been used in the literature in place of $\sA$-good moduli spaces, see for example~\cite{AHR2,FerrandPush}. They coincide when $Y$ is quasi-separated.
\end{remark}

\begin{proposition} \label{Prop:comparison between Alper and derived GMS}
    Let $X$ be a classical Artin stack and $q \colon X \to Y$ a morphism to a classical algebraic space $Y$. 
    \begin{enumerate}
        \item If $q$ defines a good moduli space for $X$, then it defines a $\mathcal{U}\sA$-good and in particular an $\sA$-good moduli space for $X$.
        \item Suppose that $X$ is finitely presented over a quasi-separated algebraic space $S$, has affine stabilizers and separated diagonal. Then $q$ defines a good moduli space if and only if it defines a $\mathcal{U}\sA$-good moduli space.
        \item Suppose that $X$ is Noetherian and has quasi-affine diagonal and that $Y$ is quasi-separated. Then if $q$ defines an $\sA$-good moduli space for $X$, it also defines a good moduli space for $X$.
    \end{enumerate}
\end{proposition}

\begin{proof}
    If $q$ defines a good moduli space for $X$, then $q$ is universally cohomologically affine by Proposition~\ref{Prop:ucd0_props}(i) and $q_\ast^\heartsuit \oO_X \simeq q_\ast \oO_X \simeq \oO_Y$ under the natural morphism. This shows part (i).

    For part (ii), suppose that $q$ defines a $\mathcal{U}\sA$-good moduli space for $X$. Then by~\cite[Theorem~6.1]{AHR2}, $q$ has affine diagonal and thus is \ucd by Proposition~\ref{Prop:ucd0_props}(iv). It follows that $q$ defines a good moduli space.

    For part (iii), if $X$ is Noetherian with quasi-affine diagonal, $q_\ast^\heartsuit$ is exact, and $Y$ is quasi-separated, then Proposition~\ref{Prop:ucd0_props}(iii) implies that $q$ is \ucd and Proposition~\ref{Prop:quasi-aff_cd_ucd_$t$-exact} implies that $q_*$ is $t$-exact (recall that $Y$ has quasi-affine diagonal since it is a quasi-separated algebraic space). Thus, we also have $q_\ast^\heartsuit \simeq q_\ast$ on $\QCoh(X)^{\heartsuit}$. It follows that $q$ defines a good moduli space for $X$.
\end{proof}

The following proposition shows that the good moduli space of a classical Artin stack must be classical. In particular, under the conditions of part (iii) in the previous proposition, our definition of good moduli space is identical to the one given by Alper.

\begin{proposition} \label{Prop:Classical stack has classical GMS}
    If $X$ is a classical Artin stack with a good moduli space $q \colon X \to Y$, then $Y$ is classical. 
\end{proposition}

\begin{proof}
    Since $q_*$ is $t$-exact and $\oO_Y \simeq q_\ast \oO_X$, it follows that $\pi_i(\oO_Y) \simeq 0$ for all $i>0$ and hence $Y$ is classical.
\end{proof}

\subsection{A discussion of different properties of morphisms and notions of good moduli space} \label{subsection:different notions of morphisms and gms}

It is evident that the definitions of good moduli space, $\mathcal{U}\sA$-good moduli space and $\sA$-good moduli space essentially differ in the property we require the morphism $q \colon X \to Y$ to satisfy, namely being \ucd, universally cohomologically affine and cohomologically affine respectively. We thus think it is fruitful to include a short discussion and examples on these properties and their relations. We thank David Rydh for helpful conversations that motivated this subsection.

Let $f \colon X \to Y$ be a qcqs morphism of Artin stacks. By Proposition~\ref{Prop:ucd0_props}, if $f$ has affine diagonal, then being \ucd is equivalent to being universally cohomologically affine. Under further assumptions, Proposition~\ref{Prop:ucd0_props} gives converse statements to this equivalence as well. The following examples show the departure between the different notions in the absence of affine diagonal.

\begin{example}
\label{example:gms1.2}
    We return to Example~\ref{example:gms1}, so we take $f \colon * \to BA$ the universal torsor for $A$ a classical abelian variety of dimension $\geq 1$. If $f_*\CC \simeq \oO_{BA}$ then again by base change we would have $\CC \simeq \Gamma(A,\oO_A)$. Since $\Gamma(A,\oO_A)$ is non-discrete, we conclude that $\oO_{BA} \to f_*\CC$ is not invertible, even though it is invertible on $\pi_0$.
\end{example}

\begin{example}
\label{example:gms2.2}
    We return to Example~\ref{example:gms2}, so we take $g \colon BA \to *$ the structure map. We claim that also in this case it holds that $\CC \to g_* \oO_{BA}$ is not invertible, even though it is invertible on $\pi_0$. To see this, take any nontrivial extension
    \[
        \Ga \to E \to A
    \]
    of $A$ by $\Ga$ (such an extension exists, for example by \cite[Thm.~3]{rosenlicht1958ExtensionsVectorGroups}). Taking deloopings gives a nontrivial map
    \[
        c \colon BA \to B^2\Ga.
    \]
    Recall that in general one has
    \[
        \Map(X,B^n\Ga) \simeq \tau_{\geq 0}(\Gamma(X,\oO_X)[n]).
    \]
    Hence
    \[
        0 \neq c \in \pi_0 \Map(BA,B^2\Ga) \simeq \pi_{-2}g_*\oO_{BA}.
    \]
\end{example}

\begin{example}
    \label{example:gms3.2}
    We return to Example~\ref{example:gms3}, so we take $E(A)$ the universal vector extension of an abelian variety $A$. Since $E(A)$ is anti-affine it holds that $\QCoh(E(A))^\heartsuit \simeq \Mod_\CC^{\heartsuit}$ by \cite[Lem.~1.1, 2.3.(i)]{rosenlicht1958ExtensionsVectorGroups} and the same argument as in Example~\ref{example:gms1}. Hence for $\varphi \colon * \to E(A)$ the universal torsor it holds that $\varphi_*^\heartsuit$ is invertible. Since $\varphi_*$ is \cd we conclude $\varphi_*\CC \simeq \oO_{BE(A)}$. Hence, also for the projection $\psi \colon BE(A) \to *$ it holds that $\psi_* \oO_{BE(A)} \simeq \CC$ since $\psi \varphi \simeq \id$.
\end{example}

Examples~\ref{example:gms1},\ref{example:gms1.2} show that being cohomologically affine is not the best-suited notion to demand for moduli space morphisms. While this was the case in the original paper~\cite{AlperGood}, the subsequent papers~\cite{AHR2,AlperLuna,FerrandPush} have instead relied on the notion of universally cohomological affineness. This is also consistent with the slogan that most well-behaved properties in algebraic geometry are invariant under base-change.

Similarly, Examples~\ref{example:gms3}, \ref{example:gms3.2} show that a reasonable definition of good moduli space should use a property that is invariant under base-change, such as \ucd. Examples~\ref{example:gms4}, \ref{example:gms3.2} then imply that using \ucd morphisms still allows for non-affine stabilizers.

Example~\ref{example:gms2} illustrates the difference between being universally cohomologically affine and \cd or \ucd, which one might argue is an interesting feature that motivates our choice to use \ucd to define good moduli spaces.
\smallskip

When the target $Y$ is a quasi-separated algebraic space and $q$ has affine diagonal, then these subtle differences disappear and all four notions coincide. In particular, this implies that our definition is consistent with the original definition of Alper~\cite{AlperGood} and its subsequent strengthening in~\cite{AHR2,AlperLuna,FerrandPush}. However, for a general good moduli space morphism and to get a robust theory, it is important to use a well-behaved notion, which for us is being \ucd.

\subsection{Properties of good moduli spaces: base-change and passage between classical and derived} We now lay the ground to show that most properties of good moduli spaces carry over from classical stacks to derived stacks. We first check that good moduli spaces behave well under base-change, which we then use to show that having a good moduli space is a property that only depends on the  underlying classical stack, one of the essential features of our definition.

\begin{lemma}
\label{Lem:GMS_base-change}
    Good moduli space morphisms are stable under base-change.
\end{lemma}

\begin{proof}
    Consider a Cartesian square
    \begin{center}
        \begin{tikzcd}
            X' \arrow[r, "f"] \arrow[d, "q'"] & X \arrow[d, "q"] \\
            Y' \arrow[r, "g"] & Y
        \end{tikzcd}
    \end{center}
    where $q$ is a good moduli space morphism. It is clear that $q'$ is ucd\textsubscript{0}. By base-change, Lemma \ref{Lem:ucd0_props}, we have
    \[ q'_*\oO_{X'} \simeq q'_*f^*\oO_X \simeq g^*q_*\oO_X \simeq g^*\oO_Y \simeq \oO_{Y'}. \qedhere \] 
\end{proof}

We move on to one of the main results of the present paper.

\begin{theorem}
\label{Thm:GMSvsGMScl}
Let $X$ be an Artin stack. 
\begin{enumerate}
    \item If $X_\cl$ admits a good moduli space $q_\cl \colon X_\cl \to Y_\cl$, then $X$ admits a good moduli space $q' \colon X \to Y$ such that $q'_\cl \simeq q_\cl$.
    \item If $X$ admits a good moduli space $q\colon X \to Y$, then $q_\cl$ is a good moduli space for $X_\cl$.
\end{enumerate}
\end{theorem}

\begin{proof}
We first assume that $X_\cl$ admits a good moduli space $q_\cl \colon X_\cl \to Y_\cl$. Note that by Proposition~\ref{Prop:Classical stack has classical GMS}, $Y_\cl$ must be classical so the statement is well-defined. By using the Postnikov tower of $X$, which exhibits $X$ as the colimit of a sequence of square-zero extensions 
\begin{align*} 
X_{\leq 0} = X_\cl \lr X_{\leq 1} \lr X_{\leq 2} \lr \ldots \lr X,
\end{align*}
it now suffices to prove that if $X$ is a stack with good moduli space $q \colon X \to Y$, and $X'$ is a square-zero extension of $X$, then $X'$ admits a natural good moduli space $q' \colon X' \to Y'$, fitting in a commutative diagram
\begin{align} \label{eq:3.5}
    \xymatrix{
    X \ar[r]^-{i} \ar[d]_-{q} & X' \ar[d]^-{q'} \\
    Y \ar[r]_-{j} & Y',
    }
\end{align}
where $Y'$ is a square-zero extension of $Y$ as well.

We thus take $X'$ to be a square-zero extension of $X$ by a quasi-coherent $\oO_{X}$-module $M \in \QCoh(X)_{[d,d]}$, determined by a morphism
\begin{align} \label{eq:3.1}
    \alpha \colon \BL_{X} \lr M.
\end{align} 
in $\QCoh(X)$.

Note that we take $M$ here to fit in an exact triangle $\oO_{X'} \to \oO_X \to M$ (which is a shift of the usual way to express square-zero extensions).

By definition and the universal property of $\BL_X$, $\alpha$ is naturally equivalent to a morphism $\tilde{\alpha} \colon X[M] \to X$ and $X'$ is defined by the pushout diagram
\begin{align} \label{eq:3.2}
    \xymatrix{
    X[M] \ar[r]^-{\pi} \ar[d]_-{\tilde{\alpha}} & X \ar[d] \\
    X \ar[r] & X',
    }
\end{align}
where $\pi$ denotes the projection map.

Since $q$ is ucd$_0$, we have $q_\ast M \in \QCoh(Y)_{[d,d]}$ and we obtain a morphism
\begin{align}
    \beta \colon \BL_Y \lr q_\ast \BL_X \xrightarrow{\ q_\ast \alpha \ } q_\ast M
\end{align}
in $\QCoh(Y)$, which corresponds to a square-zero extension $Y \to Y'$.

Similarly to the above, $\beta$ in turn corresponds to a morphism $\tilde{\beta} \colon Y[q_\ast M] \to Y$ and $Y'$ is defined by the pushout diagram
\begin{align} \label{eq:3.4}
    \xymatrix{
    Y[q_\ast M] \ar[r]^-{\pi} \ar[d]_-{\tilde{\beta}} & Y \ar[d] \\
    Y \ar[r] & Y',
    }
\end{align}
in which $\pi$ again, by abuse of notation, denotes the projection map.

The counit $\epsilon \colon q^*q_*M \to M$ induces a morphism 
\[ \tilde{q} \colon X[M] \to X[q^*q_*M] \simeq X \times_{q,Y,\pi} Y[q_*M] \to Y[q_*M] \]
 compatible with the projection maps $\pi$. 
 
 Now $q \circ \tilde{\alpha}$ is determined by the morphism
 \[ q^*\BL_Y \to \BL_X \xrightarrow{\alpha} M; \]
meanwhile $\tilde{\beta}\circ \tilde{q}$ is determined by the morphism
\[ q^*\BL_Y \to q^*q_*\BL_X \xrightarrow{q^*q_* \alpha} q^*q_*M \xrightarrow{\epsilon} M. \]
By naturality of the counit, it follows that $q \circ \tilde{\alpha} \simeq \tilde{\beta} \circ \tilde{q}$, hence that $q$ induces a morphism from the pushout square~\eqref{eq:3.2} to the pushout square~\eqref{eq:3.4}, and thus a morphism $q' \colon X' \to Y'$ fitting in a commutative diagram~\eqref{eq:3.5}. We claim that $q'$ is the desired good moduli space for $X'$.

Firstly, since $i \colon X \to X'$ is affine, $i_\ast^{\heartsuit}$ is simply the restriction $i_\ast \colon \QCoh(X)^{\heartsuit} \to \QCoh(X')^{\heartsuit}$, which fits into the commutative square
\begin{align*}
    \xymatrix{
    \QCoh(X_\cl)^{\heartsuit} \ar[r]^-{(i_\cl)_\ast} \ar[d]_-{k_\ast} & \QCoh(X'_\cl)^{\heartsuit} \ar[d]^-{k'_\ast} \\
    \QCoh(X)^{\heartsuit} \ar[r]^-{i_\ast} & \QCoh(X')^{\heartsuit} }
\end{align*}
where $k \colon X_\cl \to X,\ k' \colon X_\cl' \to X'$ denote the natural closed embeddings. Since $i_\cl$ is an equivalence, and, by definition, the same holds for $k_\ast^{\heartsuit}$ and $k_\ast'^{\heartsuit}$, the same is true for $i_\ast^{\heartsuit}$. By the same argument, $j_\ast$ is also an equivalence on the heart, for $j \colon Y \to Y'$. 

To show that $q'$ is \ucd, we may assume that $Y'$ is affine by Proposition \ref{Prop:ucd0_equivalent}. Since $j \circ q \simeq q' \circ i$ by diagram~\eqref{eq:3.5}, we get $j_\ast \circ q_\ast \simeq q'_\ast \circ i_\ast$, which by Remark \ref{rem:2.4} implies that $q'_*$ is $t$-exact, since $q_*$ is $t$-exact. Hence $q'$ is \ucd by Proposition \ref{Prop:quasi-aff_cd_ucd_$t$-exact}. 

Finally, we have a commutative diagram of distinguished triangles
\begin{align*}
    \xymatrix{
    j_\ast q_\ast M[-1] \ar[r] \ar[d] & \oO_{Y'} \ar[r] \ar[d] & j_\ast \oO_Y \simeq j_\ast q_\ast \oO_X \ar[d] \\
    q'_\ast i_\ast M[-1] \ar[r] & q'_\ast \oO_{X'} \ar[r] & q'_\ast i_\ast \oO_X.
    }
\end{align*}

Again, the equivalence $j \circ q \simeq q' \circ i$ by diagram~\eqref{eq:3.5} implies that the outer two vertical arrows are equivalences and hence the same is true for the middle arrow. This concludes (i).

Conversely, suppose that $q \colon X \to Y$ is a good moduli space morphism with $Y$ an algebraic space. By Remark \ref{Rem:coh_dim_cl}, $q_\cl$ is \ucd. Thus, to show that $q_\cl$ is a good moduli space morphism, we only need to check that the natural morphism $\oO_{Y_\cl} \to (q_{\cl})_\ast^{\heartsuit} \oO_{X_\cl}$ is an equivalence. Observe that we have a natural equivalence $(q_\cl)_\ast^{\heartsuit} \oO_{X_\cl} \simeq (q_\cl)_\ast \oO_{X_\cl}$ since $(q_\cl)_*$ is $t$-exact by Remark \ref{rem:2.4}.

Let $k \colon X_\cl \to X$ and $\ell \colon Y_\cl \to Y$ denote the natural closed embeddings so that $q \circ k \simeq \ell \circ q_\cl$. Using the $t$-exactness of $q_\ast$, we have
\begin{align*}
    \ell_\ast \oO_{Y_\cl} \simeq \pi_0(\oO_Y) \simeq \pi_0 (q_\ast \oO_X) \simeq q_\ast \pi_0(\oO_X) \simeq q_\ast k_\ast \oO_{X_\cl} \simeq \ell_\ast (q_\cl)_\ast \oO_{X_\cl}.
\end{align*}
Since $\ell_\ast$ is an equivalence on the heart, this concludes the proof.
\end{proof}

\begin{remark}
    The preceding theorem is consistent with the usual analogy between the closed embedding $X_\cl \to X$ for a derived stack $X$ and the closed embedding $X_{\mathrm{red}} \to X$ of the canonical reduced substack of a classical algebraic stack $X$. Namely, under the additional assumption that the good moduli space morphism is \cd, \cite[Corollary~5.7]{AlperLocalProp} implies that a classical stack $X$ admits an $\sA$-good moduli space if and only if $X_{\mathrm{red}}$ does. An alternative way to see this when $X$ is Noetherian is to apply the argument in the preceding proof to the sequence of square-zero closed embeddings $X_{\mathrm{red}} = X_0 \to X_1 \to \ldots \to X_N = X$, where $X_n = \Spec_X ({\oO_X / I^n})$, $I$ is the ideal sheaf of $X_{\mathrm{red}}$ in $X$ and $N$ is large enough. Hence, it should not come as a surprise that the same holds for a derived stack $X$ and its classical truncation $X_\cl$.
\end{remark}

\subsection{Properties of good moduli spaces: universality} Before showing universality of good moduli spaces, we need the following lemma.

\begin{lemma}
\label{Lem:s0-pushout}
    Let $q \colon X \to Y$ be a morphism of Artin stacks, and $M \in \QCoh(X)_{\geq 0}$. Then the commutative square
    \begin{center}
        \begin{tikzcd}
            X[M] \arrow[d, "\tilde{q}"] \arrow[r, "\pi"] & X  \arrow[d, "q"] \\
            Y[q_*M] \arrow[r, "\pi"] & Y
        \end{tikzcd}
    \end{center}
    where $\pi$ denotes the projections, is a pushout.
\end{lemma}

\begin{proof}
    By the pasting lemma for pushouts, it suffices to show that
      \begin{center}
        \begin{tikzcd}
            X \arrow[r, "\zeta"] \arrow[d, "q"] & X[M] \arrow[d, "\tilde{q}"] \\
            Y \arrow[r, "\zeta"] & Y[q_*M]
        \end{tikzcd}
    \end{center}  
    is a pushout, where the horizontal arrows are the zero sections. Let $f \colon Y \to T$ be given. Then the space of pairs $(g,\sigma)$, where $g \colon X[M] \to T$ and $\sigma \colon g \circ \zeta \simeq f \circ q$, is equivalent to 
    \[ \Map((f\circ q)^* \BL_T,M) \simeq \Map(f^*\BL_T,q_*M), \]
    which in turn is equivalent to the space of pairs $(h,\tau)$, where $h \colon Y[q_*M] \to T$ and $\tau \colon h \circ \zeta \simeq f$. 
\end{proof}

\begin{theorem} \label{thm:universality of gms}
    Good moduli spaces are universal for maps to algebraic spaces.
\end{theorem}

\begin{proof}
    Let $Z$ be an algebraic space and $q \colon X \to Y$ a good moduli space. We will show that 
    \[ q^* \colon \Map(Y,Z) \to \Map(X,Z) \]
    is an equivalence, by considering the Postnikov towers of $X$ and $Y$. Thus, write $X = \varinjlim X_{\leq n}$ and $Y = \varinjlim Y_{\leq n}$. By construction and the proof of Theorem~\ref{Thm:GMSvsGMScl}, we have good moduli spaces $q_{\leq n} \colon X_{\leq n} \to Y_{\leq n}$ which are compatible with the natural square-zero extensions $X_{\leq n} \to X_{\leq n+1}$ and $Y_{\leq n} \to Y_{\leq n+1}$. 

    By the pasting law for pushouts, Lemma \ref{Lem:s0-pushout} implies that 
      \begin{center}
        \begin{tikzcd}
            X_{\leq n} \arrow[d, "q_{\leq n}"] \arrow[r] & X_{\leq n+1} \arrow[d, "q_{\leq n+1}"] \\
            Y_{\leq n} \arrow[r] & Y_{\leq n+1}
        \end{tikzcd}
    \end{center}     
    is a pushout, for each $n$. It follows that
    \begin{align*}
        \Map(Y,Z) &\simeq \varprojlim \Map(Y_{\leq n}, Z) \\
        &\simeq \varprojlim (\Map(X_{\leq n},Z) \times_{\Map(X_{\leq 0},Z)} \Map(Y_{\leq 0},Z)) \\
        &\simeq \varprojlim \Map(X_{\leq n},Z) \\
        & \simeq \Map(X,Z)
    \end{align*}
    where the third line follows from the classical case, \cite[Theorem~3.12]{AHR2}, and Proposition~\ref{Prop:comparison between Alper and derived GMS}(i), and the adjoint property of $(-)_\cl = (-)_{\leq 0}$.
\end{proof}

We extend the notation $X_{\leq n}$ by putting $X_{\leq \infty} \coloneqq X$. The following statement is now immediate.
 
\begin{corollary}
\label{Cor:GMS_cl}
    Let $X$ be an Artin stack, and $0 \leq n \leq m \leq \infty$. Then the space of good moduli spaces for $X_{\leq m}$ is equivalent to the space of good moduli spaces for $X_{\leq n}$, which is either empty or contractible. In particular, any good moduli space $q_m$ for $X_{\leq m}$ is induced from a good moduli space $q_n$ for $X_{\leq n}$ via the procedure in the proof of Theorem \ref{Thm:GMSvsGMScl}, which gives us a pushout
    \begin{center}
        \begin{tikzcd}
            X_{\leq n} \arrow[d, "q_n"] \arrow[r] & X_{\leq m} \arrow[d, "q_m"] \\
            Y_{\leq n} \arrow[r] & Y_{\leq m}
        \end{tikzcd}
    \end{center}
    such that $Y_{\leq n} \simeq (Y_{\leq m})_{\leq n}$.
\end{corollary}

\subsection{Properties of good moduli spaces continued}
We now list several properties that mirror the classical case (cf.~\cite[Theorem~4.16]{AlperGood}). First a small lemma.

\begin{lemma}
\label{lem:fpqc_connective}
    Let $f \colon Y' \to Y$ be a faithfully flat morphism of Artin stacks. If $M \in \QCoh(Y)$ is such that $f^*M$ is connective, then $M$ is connective.
\end{lemma}

\begin{proof}
    By definition of the $t$-structure on $\QCoh(Y)$, we may assume that $Y$ is affine. By taking a smooth cover of $Y'$, we may assume that $Y'$ is a disjoint union of affine schemes. We then further reduce to a finite disjoint union of affines, hence to the case where $Y'$ is affine as well.

    The exact sequence $\tau_{\geq 0} M \to M \to \tau_{<0} M$ induces the exact sequence
    \[ \tau_{\geq 0}f^* M \to f^*M \to \tau_{<0} f^*M,  \]
    since $f$ is flat hence $f^*$ commutes with the truncation functors. If $f^*M$ is connective, then $\tau_{<0} f^*M \simeq f^* \tau_{<0} M \simeq 0$, hence $\tau_{<0} M \simeq 0$ since $f$ is faithfully flat. The claim follows.
 \end{proof}

\begin{lemma} \label{lem:gms_prop}
    Let $q \colon X \to Y$ be a good moduli space and $f \colon X \to S$ a morphism from $X$ to an algebraic space $S$.
    \begin{enumerate}
        \item $q$ is surjective and universally closed.
        \item Being a good moduli space is fpqc-local on the target. 
        \item Let $A$ be a sheaf of quasi-coherent $\oO_X$-algebras. Then the morphism $\Spec_X A \to \Spec_Y q_*A$ is a good moduli space.
        \item Let $k$ be an algebraically closed field. For $x_1,x_2 \in \lvert X(k) \rvert$, the relation $\sim$ for which $x_1 \sim x_2$ if $\overline{\{x_1\}} \cap \overline{\{x_2\}} \neq \emptyset$ is an equivalence relation such that $\lvert Y(k) \rvert$ is the quotient of $\lvert X(k) \rvert$ by $\sim$. 
        \item If $X$ is flat over $S$, then the same is true for the induced morphism $Y \to S$.
    \end{enumerate}
\end{lemma}

\begin{proof}
    (i) and (iv) are a consequence of Theorem~\ref{Thm:GMSvsGMScl}(ii) and the fact that a good moduli space for $X_\cl$ is an $\sA$-good moduli space for $X_\cl$ by Proposition~\ref{Prop:comparison between Alper and derived GMS}. Thus \cite[Theorem~4.16]{AlperGood} applies.
    
    For (ii), suppose that we have a Cartesian diagram
    \begin{center}
        \begin{tikzcd}
            X' \arrow[d, "q'"] \arrow[r, "g"] & X \arrow[d, "q"] \\
            Y' \arrow[r, "f"] & Y,
        \end{tikzcd}
    \end{center}
    where $q'$ is a good moduli space and $f$ is faithfully flat and quasi-compact. Note that this implies that $q$ is qcqs. We need to show that $q$ is a good moduli space. By Proposition \ref{Prop:ucd0_equivalent} and Proposition \ref{Prop:quasi-aff_cd_ucd_$t$-exact}, it suffices to check that $q$ is \cd and $q_\ast \oO_X \simeq \oO_Y$. Let $M \in \QCoh(X)^{\heartsuit}$ be given. By \cite[Lem.~A.1.3]{HalpernleistnerMapping}, it holds that $f^*q_*M \simeq q'_*g^*M$, since $f$ is flat and $M$ is bounded above. By assumption and the flatness of $g$, it follows that $f^*q_*M \in \QCoh(Y')_{\geq 0}$. By Lemma \ref{lem:fpqc_connective}, it follows that $q_*M \in \QCoh(Y)_{\geq 0}$. This shows that $q$ is \cd. And again by \cite[Lem.~A.1.3]{HalpernleistnerMapping}, we have
    \[ f^*q_*\oO_X \simeq q'_*g^*\oO_X \simeq \oO_{Y'} \simeq f^\ast \oO_Y \]
    since $q'_* \oO_{X'} \simeq \oO_{Y'}$. Since $f$ is faithfully flat, $f^*$ is conservative, hence $q_*\oO_X \simeq \oO_Y$.

    For (iii), consider the commutative diagram
    \begin{align*}
        \xymatrix{
        \Spec_X A \ar[r]^-{i} \ar[d]_-{q'} & X \ar[d]^-{q} \\
        \Spec_Y q_\ast A \ar[r]_-j & Y.
        }
    \end{align*}
    The morphism $j$ is affine and thus has affine diagonal. Moreover $j \circ q' \simeq q \circ i$ is \ucd, since $i$ is \ucd as an affine morphism and $q$ is \ucd. It follows by Proposition~\ref{Prop:cancellation property} that $q'$ is \ucd. 
    
    Let $F$ be the cofiber of the natural morphism $\oO_{\Spec_Y q_\ast A} \to q'_\ast \oO_{\Spec_X A}$. Since $j_\ast q'_\ast \oO_{\Spec_X A} \simeq q_\ast i_\ast \oO_{\Spec_X A} \simeq q_\ast A \simeq j_\ast \oO_{\Spec_Y q_\ast A}$, we deduce that $j_\ast F \simeq 0$. But for any affine morphism $f \colon S \to T$, the functor $f_*\colon \QCoh(S)\to \QCoh(T)$ is the forgetful functor from $\oO_S$-algebras in $\QCoh(T)$ to $\oO_T$-algebras in $\QCoh(T)$, hence conservative. It follows that $F \simeq 0$ since $j_*F \simeq 0$.

    For (v), by universality of good moduli spaces, we have a natural morphism $g \colon Y \to S$. As above, $Y_\cl$ is a good and hence an $\sA$-good moduli space for $X_\cl$, which must be flat over $S$ by \cite[Theorem~4.16]{AlperGood}. It remains to check that the natural morphism $g_\cl^\ast \pi_n(\oO_S) \simeq \pi_n(\oO_S) \otimes_{\pi_0(\oO_S)} \pi_0(\oO_Y) \to \pi_n(\oO_Y)$ is an equivalence for all $n > 0$. By the flatness of $X$ over $S$, this is the case for the morphism $\pi_n(\oO_S) \otimes_{\pi_0(\oO_S)} \pi_0(\oO_X) \to \pi_n(\oO_X)$. We then have by definition $\pi_n(\oO_X) \simeq \pi_n(\oO_S) \otimes_{\pi_0(\oO_S)} \pi_0(\oO_X) \simeq q_\cl^\ast g_\cl^\ast \pi_n(\oO_S)$. Applying the pushforward and noting that $\pi_n(\oO_Y) \simeq (q_\cl)_\ast \pi_n(\oO_X)$ for all $n \geq 0$ as $\pi_0(\oO_Y)$-modules by the proof of Theorem~\ref{Thm:GMSvsGMScl}, together with $(q_\cl)_\ast q_\cl^\ast \simeq \id$ by the projection formula, we obtain
    $$ \pi_n(\oO_Y) \simeq (q_\cl)_\ast \pi_n(\oO_X) \simeq (q_\cl)_\ast q_\cl^{\ast} g_\cl^\ast \pi_n(\oO_S) \simeq g_\cl^\ast \pi_n(\oO_S),$$
    which concludes the proof.
\end{proof}

\begin{remark}
    Finite generation and other similar properties for good moduli spaces, as stated in \cite[Theorem~4.16]{AlperGood}, hold in the derived setting as well in the obvious way. Since the translation to our context is immediate for such properties, we omit them in the interest of space.
\end{remark}

\subsection{Closed substacks}
Throughout this subsection, let $q \colon X \to Y$ be a good moduli space.
\begin{definition} \label{def: saturated closed substack}
     The \emph{scheme-theoretic image} of a closed immersion $Z \to X$ is defined as $q(Z) \coloneqq \Spec_Y q_* \oO_Z$. For a closed immersion $Z' \to Y$, its \emph{scheme-theoretic pre-image} is defined as $q^{-1}(Z') = Z' \times_Y X$. We say that a closed substack $Z \to X$ is \emph{saturated} if the natural morphism $Z \to q^{-1}(q(Z))$ is an equivalence. $q^{-1}(q(Z)) = \Spec_X (q^\ast q_\ast \oO_Z)$ is called the \emph{saturation} of the closed substack $Z$ and denoted by $Z_{\mathrm{sat}}$.
\end{definition}

\begin{remark}
    Observe that $q(Z) \to Y$ is a closed immersion, and that $q(Z)_\cl$ is the classical scheme-theoretic image of $Z_\cl \to X_\cl$ under $q_\cl$. By definition, it holds that $q(X) \simeq Y$.
\end{remark}

\begin{remark}
    In general, scheme-theoretic images are subtle in derived algebraic geometry. However, in the case of a good moduli space $q \colon X \to Y$, the fact that $q_\ast \oO_X \simeq \oO_Y$ and that $q_*$ is $t$-exact allows us to give an easy definition.
\end{remark}

\begin{corollary}
    Let $Z \to X$ be a closed immersion. The morphism $Z \to q(Z)$ defines a good moduli space for $Z$.
\end{corollary}

\begin{proof}
    This follows from Lemma~\ref{lem:gms_prop}(iii) applied to $A = \oO_Z$.
\end{proof}

\begin{lemma}
\label{Lem:GMS_image-pull}
    Let $Y' \to Y$ be a closed immersion, and let $q' \colon X' \to Y'$ be the pullback of $q$. Then the natural map $Y' \to q(X')$ is an equivalence.
\end{lemma}

\begin{proof}
    This is immediate from Lemma \ref{Lem:GMS_base-change}.
\end{proof}

The following proposition is an application of the projection formula for \ucd morphisms.

\begin{proposition}
    Let $Z_1 \to X$ and $ Z_2 \to X$ be closed immersions. For $i=1,2$, let $F_i$ denote the fiber of the morphism $\oO_X \to \oO_{Z_i}$.
    
    Then the natural morphism 
    \begin{align} \label{eq: loc 3.8}
        q(Z_1 \times_X Z_2) \to q(Z_1) \times_Y q(Z_2)
    \end{align} 
    is an equivalence if and only if the natural morphism $\varphi \colon q_\ast {F_1} \otimes_{\oO_Y} q_\ast {F_2} \to q_\ast ( {F_1} \otimes_{\oO_X} {F_2} )$ is an equivalence. In particular, this is the case if $Z_2$ is saturated.
\end{proposition}

\begin{proof}
    Let $C_i$ be the cofiber of the morphism $q^\ast q_\ast F_i \to F_i$ for $i=1,2$. Note that by the projection formula we have $q_\ast C_i = 0$.
    
    We first show that $q_\ast ( {F_1} \otimes_{\oO_X} {F_2} ) \simeq q_\ast {F_1} \otimes_{\oO_Y} q_\ast {F_2}$ is equivalent to $q_\ast(\oO_{Z_1} \otimes_{\oO_X} C_2) = 0$.

    Consider the exact triangle $q^\ast q_\ast F_2 \to F_2 \to C_2$. Applying the functor $q_\ast(F_1 \otimes_{\oO_X} ( - ) )$ and using the projection formula, the cofiber of $\varphi$ is $q_\ast(F_1 \otimes_{\oO_X} C_2)$.

    There is an exact triangle
    $$q_\ast(F_1 \otimes_{\oO_X} C_2) \lr q_\ast C_2 \lr q_\ast(\oO_{Z_1} \otimes_{\oO_X} C_2).$$

    Since $q_\ast C_2 = 0$, we see that $q_\ast(\oO_{Z_1} \otimes_{\oO_X} C_2) = 0$ if and only if $q_\ast(F_1 \otimes_{\oO_X} C_2) = 0$ if and only if $\varphi$ is an equivalence.
    
    Now, consider the exact triangle $F_2 \to \oO_X \to \oO_{Z_2}$. Applying $q_\ast$ and tensoring with $q_\ast \oO_{Z_1}$ gives an exact triangle
    \[
        q_\ast \oO_{Z_1} \otimes_{\oO_Y} q_\ast F_2 \lr q_\ast \oO_{Z_1} \lr q_*\oO_{Z_1} \otimes_{\oO_Y} q_*\oO_{Z_2}. 
    \]
    
    On the other hand, tensoring with $\oO_{Z_1}$ and then applying $q_\ast$ gives an exact triangle
    \[
        q_\ast (\oO_{Z_1} \otimes_{\oO_X}  F_2) \lr q_\ast \oO_{Z_1} \lr q_*(\oO_{Z_1} \otimes_{\oO_X} \oO_{Z_2}). 
    \]

    It thus suffices to prove that the natural morphism 
    \[
    \varphi \colon q_\ast \oO_{Z_1} \otimes_{\oO_Y} q_\ast F_2 \lr q_\ast (\oO_{Z_1} \otimes_{\oO_X}  F_2) 
    \]
    is an equivalence if and only if $q_\ast(\oO_{Z_1} \otimes_{\oO_X} C_2) = 0$.

    Consider the exact triangle $q^\ast q_\ast F_2 \to F_2 \to C_2$. Then, applying the functor $q_\ast(\oO_{Z_1} \otimes_{\oO_X} ( - ) )$ and using the projection formula, the cofiber of $\varphi$ is exactly $q_\ast(\oO_{Z_1} \otimes_{\oO_X} C_2)$, as we want.

    If $Z_2$ is saturated, then by definition $\oO_{Z_2} \simeq q^*q_*\oO_{Z_2}$ and hence $C_2 \simeq 0$. Thus, the condition holds and we are done.
\end{proof}

\begin{remark}
    The proof of the proposition shows that we could have equivalently required that $q_\ast(\oO_{Z_1} \otimes_{\oO_X} C_2) = 0$ or $q_\ast(F_1 \otimes_{\oO_X} C_2) = 0$ as the condition in the statement.
\end{remark}

\begin{remark}
    The above proposition shows that there is a difference in behavior regarding scheme-theoretic images of closed substacks and their intersection between the classical and derived setting. Namely, the morphism~\eqref{eq: loc 3.8} is always an isomorphism when everything is classical (including the closed immersions). 

    For a concrete example that shows that the condition of the proposition is not automatic, consider $X = B\GG_m$ and let $Z_1 = [\Spec k[\epsilon_1] / \GG_m]$, $Z_2 = [\Spec k[\epsilon_2] / \GG_m]$ be two closed substacks of $X$ where $\epsilon_i$ are in cohomological degree $-1$, satisfy $\epsilon_i^2 = 0$ and the $\GG_m$-weight of $\epsilon_i$ is equal to $(-1)^i$, over an algebraically closed field $k$ of characteristic $0$. Since $X$ is classical Noetherian with affine diagonal, its good moduli space as a derived stack is the same as its $\sA$-good moduli space and hence is given by the canonical morphism $q \colon X = B\GG_m \to *=Y$. As we will see in Subsection~\ref{subsection:group action examples}, the good moduli spaces of $Z_1$ and $Z_2$ are obtained by taking invariants. It follows that the good moduli spaces of $Z_1$ and $Z_2$ are also equivalent to the point $*$. In particular, we have that $q(Z_1) \times_Y q(Z_2) \simeq *$.
    
    On the other hand, $Z_1 \times_X Z_2 = [\Spec k[\epsilon_1, \epsilon_2] / \GG_m]$. Considering $k[\epsilon_1, \epsilon_2]$ as an $\oO_{Z_1 \times_X Z_2}$-module with the given $\GG_m$-weights, we see that $q_\ast k[\epsilon_1, \epsilon_2] = k[\epsilon_1, \epsilon_2]^{\GG_m} = k \oplus k \epsilon_1 \epsilon_2$, which is not equal to $k$. Thus the morphism~\eqref{eq: loc 3.8} cannot be an equivalence.
\end{remark}

\subsection{Gluing good moduli spaces} 
Suppose that $X$ is an Artin stack with good moduli space $q \colon X \to Y$. Let $U \subseteq \lvert X \rvert$ be an open subset. Consider the image $q (U)$ inside $\lvert Y \rvert$. In analogy with the case of a closed substack, we say that $U$ is \textit{saturated} if $U$ is the preimage of $q(U)$ under $q \colon \lvert X \rvert \to \lvert Y \rvert$. An open substack $X' \subseteq X$ is \emph{saturated} if $\lvert X'\rvert \subseteq \lvert X \rvert$ is saturated. 

\begin{proposition}
    Let $q \colon X \to Y$ be a good moduli space. 
    \begin{enumerate}
        \item If $Y' \subseteq Y$ is an open substack, then $q^{-1}(Y') := Y' \times_Y X$ is a saturated open substack of $X$. 
        \item If $X' \subseteq X$ is a saturated open substack, then $q(\lvert X' \rvert)$ is open, and thus induces an open substack $q(X') \subseteq Y$ such that $X'\to q(X')$ is a good moduli space.
    \end{enumerate}
\end{proposition}

\begin{proof}
    This is clear from the definition.
\end{proof}

The following proposition describes the conditions under which good moduli spaces on an open cover can be glued (cf. \cite[Proposition~7.9]{AlperGood}).

\begin{proposition}
    Let $X$ be an Artin stack, and $\lbrace U_i \rbrace_{i \in I}$ an open cover by substacks such that for each index $i$ there exists a good moduli space $q_i \colon U_i \to V_i$. Then there exists a good moduli space $q \colon X \to Y$ and open substacks $Y_i \subseteq Y$ such that $Y_i \simeq V_i$ as $Y$-stacks and we have Cartesian squares
    \begin{align} \label{eqn:loc 3.8}
        \xymatrix{
        U_i \ar[d]_-{q_i} \ar[r] & X \ar[d]^-{q} \\
        V_i\simeq Y_i \ar[r] & Y 
        }
    \end{align}
    if and only if for any two indices $i,j \in I$, the intersection $U_{ij} := U_i \cap U_j = U_i \times_X U_j$ is a saturated open substack of $U_i$ and $U_j$.
\end{proposition}

\begin{proof}
    The only if direction follows from the preceding proposition.

    For the converse, for $I = (i_1,\dots,i_n)$, let $U_I$ be the intersection $U_{i_1} \times_X \cdots \times_X U_{i_n}$. Then $U_I \subset X$ is saturated, hence admits a good moduli space $V_I$. For $J$ another (ordered) indexing set $J = (j_1,\dots,j_m)$, write $J \leq I$ if $m \leq n$ and $\{j_1,\dots,j_m\} \subset \{i_1,\dots,i_n\}$. Then by Theorem \ref{thm:universality of gms}, for any $J \leq I$ we obtain a pullback square
    \begin{center}
        \begin{tikzcd}
            U_I \arrow[d] \arrow[r] & U_J \arrow[d] \\
            V_I \arrow[r] & V_J,
        \end{tikzcd}
    \end{center}
    and a diagram $I \mapsto (U_I \to V_I)$. Then the colimit $q \colon X \to Y$ of this diagram is a good moduli space by Lemma \ref{lem:gms_prop}, and it is clear that $q$ has the desired properties.
\end{proof}

\subsection{Descent of \'{e}tale morphisms} We proceed to describe conditions that ensure that \'{e}tale morphisms between stacks descend to \'{e}tale morphisms between their good moduli spaces.

Given a commutative diagram
\begin{align*}
        \xymatrix{
        X' \ar[r]^-{f} \ar[d]_-{q'} & X \ar[d]^-{q} \\
        Y' \ar[r]_-{g} & Y,
        }
\end{align*}
where $q,q'$ are good moduli spaces, we say that $f$ is \emph{strongly \'{e}tale} if the induced morphism $g$ is \'{e}tale and the square is Cartesian.

The following proposition is the translation of~\cite[Theorem~3.14]{AHR2} to the derived setting. We thank David Rydh for suggesting the following argument.

\begin{proposition}
    Consider a commutative diagram
    \begin{align} \label{eq:truncated Cartesian sq}
        \xymatrix{
        X' \ar[r]^-{f} \ar[d]_-{q'} & X \ar[d]^-{q} \\
        Y' \ar[r]_-{g} & Y,
        }
    \end{align}
    where $q',q$ are good moduli spaces. Let $x \in |X'|$ be a closed point such that:
    \begin{enumerate}
        \item $f$ is \'{e}tale and $f_\cl$ is representable in a neighborhood of $x$,
        \item $f(x) \in |X|$ is closed, and
        \item $f_\cl$ induces an isomorphism of stabilizer groups at $x$.
    \end{enumerate}
    Then there exists a saturated open neighborhood $U' \subseteq X'$ of $x$ such that $f|_{U'}$ is strongly \'{e}tale.
\end{proposition}

\begin{proof}
    By \cite[Theorem~3.14]{AHR2}, there exists a saturated open neighborhood $U_\cl' \subseteq X'_\cl$ of $x$ such that $f_\cl|_{U_\cl'}$ is strongly \'{e}tale. By topological invariance, $U_\cl'$ is the classical truncation of a saturated open neighborhood $U' \subseteq X'$ of $x$. We now check that $f|_{U'}$ is strongly \'{e}tale. To simplify notation, without loss of generality, we assume from now on that $U' = X'$.

    For each $n \geq 0$, write $q_{\leq n} \colon X_{\leq n} \to Y_{\leq n}$ for the corresponding truncation so that $X_{\leq n+1}$ and $Y_{\leq n+1}$ are square-zero extensions of $X_{\leq n}$ and $Y_{\leq n}$ by modules $M_{n+1}$ and $(q_{\leq n})_\ast M_{n+1}$ respectively, by the proof of Theorem~\ref{Thm:GMSvsGMScl}. The same notation applies for $X'$ and $Y'$.

    We will use induction on $n \geq 0$ to prove that $f_{\leq n+1}$ is strongly \'{e}tale. So assume that the truncation $f_{\leq n}$ is strongly \'{e}tale. Note that $n=0$ corresponds to the classical case, which we have established above.

    Since $f$ is \'{e}tale, $X'_{\leq n+1}$ is a square-zero extension of $X'_{\leq n}$ by the module $f_{\leq n}^\ast M_{n+1}$. Thus, the four mentioned square-zero extensions are determined by morphisms
    \begin{align*}
        \BL_{X'_{\leq n}} \to f_{\leq n}^\ast M_{n+1}, \ \  & \BL_{X_{\leq n}} \to M_{n+1} \\
        \BL_{Y'_{\leq n}} \to (q'_{\leq n})_\ast f_{\leq n}^\ast M_{n+1}, \ \  & \BL_{Y_{\leq n}} \to (q_{\leq n})_\ast M_{n+1}. 
    \end{align*}

    Since $f_{\leq n}$ is strongly \'{e}tale, by base-change we have $$(q'_{\leq n})_\ast f_{\leq n}^\ast M_{n+1} \simeq g_{\leq n}^\ast (q_{\leq n})_\ast M_{n+1},$$
    which shows that $g_{\leq n+1}$ must be \'{e}tale.

    It remains to check that the truncation of the induced commutative square of the form~\eqref{eq:truncated Cartesian sq} up to degree $n+1$ is also Cartesian. 

    Consider the natural morphism $h_{\leq n+1} \colon X'_{\leq n+1} \to Y'_{\leq n+1} \times_{Y_{\leq n+1}} X_{\leq n+1}$. We know that $(h_{\leq n+1})_\cl$ is an isomorphism, so it suffices to check that $h_{\leq n+1}$ is \'{e}tale. We have a factorization 
    $$ f_{\leq n+1} \colon X'_{\leq n+1} \xrightarrow{ \ h_{\leq n+1} \ } Y'_{\leq n+1} \times_{Y_{\leq n+1}} X_{\leq n+1} \xrightarrow{ \ g'_{\leq n+1} \ } X_{\leq n+1},$$
    where $g'_{\leq n+1}$ is the base-change of $g_{\leq n+1}$. Since $f_{\leq n+1}, g'_{\leq n+1} $ are both \'{e}tale, the same holds for $h_{\leq n+1}$ and the induction is complete.

    This implies the \'{e}taleness of $g$, since it is equivalent to $g_\cl$ being \'{e}tale together with the natural morphism $g_\cl^\ast \pi_n(\oO_Y) \to \pi_n(\oO_Y')$ being an equivalence for all $n \geq 0$. Since both the latter condition and the condition that the square~\eqref{eq:truncated Cartesian sq} is Cartesian can be checked on the truncation up to degree $n$, the proof is complete by the above.
\end{proof}

\begin{remark}
    In future work, we plan to investigate the descent of general properties of morphisms between Artin stacks to their good moduli spaces and prove a derived analogue of a general Luna's fundamental lemma in the style of \cite[Theorem~A]{LunaRydh}. 
\end{remark}

\section{Applications and examples} \label{sec:4}

This section focuses on applications and examples of the theory developed thus far. These include a derived \'{e}tale slice theorem for good moduli spaces, an upgrade of the stabilizer reduction algorithm for derived stacks \cite{HRS}, derived geometric invariant theory, an existence criterion and quotients by group actions.

\subsection{A (fully) derived \'{e}tale slice theorem for good moduli spaces} Let $X_\cl$ be a classical Artin stack with an $\sA$-good moduli space $q_\cl \colon X_\cl \to Y_\cl$ such that $X_\cl$ has affine stabilizers, separated diagonal and is of finite presentation over a quasi-separated algebraic space $S$. For simplicity, we assume that $S=\Spec k$ with $k$ an algebraically closed field, even though the results of \cite{AHR2} treat more general bases as well. Let $x \in X_{\cl}$ be a closed point with stabilizer $G_x$ (which has to be reductive, since $X_\cl$ admits a good moduli space). 

The classical \'{e}tale slice theorem for good moduli spaces~\cite[Theorem~4.12]{AlperLuna},~\cite[Theorem~6.1]{AHR2} asserts that there exists an \'{e}tale morphism $\Phi_\cl \colon [U_\cl / G_x] \to X_\cl$, where $U_\cl$ is a classical affine scheme with a $G_x$-action, $u \in U_\cl$ is fixed by $G_x$ and maps to $x$ via $\Phi_\cl$, fitting in a Cartesian square
    \begin{align*}
        \xymatrix{
        [U_\cl / G_x] \ar[d]_-{q'} \ar[r]^-{\Phi_\cl} & X_\cl \ar[d]^-{q} \\
        U_\cl \git G_x \ar[r]^-{\phi_\cl} & Y_\cl,
        }
    \end{align*}
where the map $U_\cl \git G_x \to Y_\cl$ is \'{e}tale.

We can now obtain a derived version of this statement, which includes good moduli spaces. Derived versions of the \'{e}tale slice theorem have previously also been obtained in \cite[Lemma~4.2.5]{HL} and \cite[Theorem~1.13, Proposition~6.1]{FerrandPush}.

\begin{theorem} \label{thm:derived slice thm}
    Let $X$ be an Artin stack with a good moduli space $q\colon X \to Y$ such that $X_\cl$ satisfies the conditions listed above. Let $x \in |X|$ be a closed point and write $G_x$ for the classical stabilizer group of $x \in |X_\cl|$. 
    
    Then there exists an \'{e}tale morphism $\Phi \colon [U / G_x] \to X$, where $U$ is an affine scheme with a $G_x$-action, $u \in U$ is fixed by $G_x$ and maps to $x$ via $\Phi$, fitting in a Cartesian square
    \begin{align*}
        \xymatrix{
        [U / G_x] \ar[d]_-{q'} \ar[r]^-{\Phi} & X \ar[d]^-{q} \\
        U \git G_x \ar[r]^-{\phi} & Y,
        }
    \end{align*}
    where $U \git G_x$ denotes the good moduli space of $[U / G_x]$ and the morphism $U \git G_x \to Y$ is \'{e}tale.
\end{theorem}

\begin{proof}
    By~\cite[Theorem~4.12]{AlperLuna} and~\cite[Theorem~6.1]{AHR2}, there exists a classical morphism $\Phi_\cl \colon [U_\cl / G_x] \to X_\cl$, which is \'{e}tale, stabilizer-preserving and fits in a Cartesian square 
    \begin{align} \label{diag:loc 4.1}
        \xymatrix{
        [U_\cl / G_x] \ar[d]_-{q'_\cl} \ar[r]^-{\Phi_\cl} & X_\cl \ar[d]^-{q_\cl} \\
        U_\cl \git G_x \ar[r]^-{\phi_\cl} & Y_\cl,
        }
    \end{align}
    where the morphism $[U_\cl/G_x] \to U_\cl \git G_x$ is a good moduli space morphism and $U_\cl \git G_x \to Y_\cl$ is \'{e}tale (here we have implicitly used Proposition~\ref{Prop:comparison between Alper and derived GMS}(ii)).

    By topological invariance \cite[Lemma~4.11]{HRS}, $\phi_\cl \colon U_\cl \git G_x \to Y_\cl$ is the classical truncation of an \'{e}tale morphism $\phi \colon V \to Y$. Since the classical truncation of the fiber product $V \times_Y X$ is the quotient stack $[U_\cl / G_x]$, by the argument used in the proof of \cite[Proposition~4.13]{HRS}, it follows that $V \times_Y X$ is a quotient stack of the form $[U / G_x]$ for some affine scheme $U$ with a $G_x$-action, whose classical truncation is the $G_x$-scheme $U_\cl$.

   We thus have a Cartesian square 
    \begin{align*}
        \xymatrix{
        [U / G_x] \ar[d]_-{q'} \ar[r]^-{\Phi} & X \ar[d]^-{q} \\
        V \ar[r]_-{\phi} & Y,
        }
    \end{align*}
    whose classical truncation is the square~\eqref{diag:loc 4.1}.

    By Lemma~\ref{Lem:GMS_base-change}, $q'$ is a good moduli space morphism, as it is a base-change of the good moduli space morphism $q$, and the proof is complete.
\end{proof}

\begin{remark}
    In Subsection~\ref{subsection:group action examples}, we give a hands-on description of the good moduli space $U \git G_x$ of $[U/G_x]$.
\end{remark}

\subsection{Derived saturated blowups, stabilizer reduction and partial desingularization} Let $X$ be a derived Artin stack with a good moduli space $q \colon X \to Y$.

Stabilizer reduction refers to the process of producing a canonical stack $\tilde{X}$ together with a canonical morphism $\pi \colon \tilde{X} \to X$, which resolves the stackiness of $X$, meaning that $\tilde{X}$ is Deligne--Mumford and $\pi$ alters the original stack as little as possible. This has successfully been carried out within classical algebraic geometry originally for smooth quotient stacks obtained through Geometric Invariant Theory (GIT) in \cite{Kirwan} and then more generally for singular classical stacks in the presence of a good moduli space in \cite{EdidinRydh, Sav}. 

Subsequently, a derived stabilizer reduction procedure was developed in \cite{HRS}. However, the derived version only involves classical good moduli spaces in its statement and results. We can now give a unified, derived upgrade of the stabilizer reduction algorithm at the level of good moduli spaces as well.
\smallskip

We begin by recalling the definition of derived saturated blowups and setting up notation. These form the basic operation which is iteratively used in the stabilizer reduction process.  

\begin{definition}
    Let $X$ be an Artin stack. The \emph{saturated projective spectrum} of a quasi-coherent, graded $\sO_X$-algebra $A$ is defined as the largest open substack $\Proj^q_X A \subseteq \Proj_X A$ for which the morphism $\Proj^q_X A \to \Proj_X q^*q_*A$ induced by $q^*q_*A \to A $ is well-defined. In particular, we have a commutative diagram
    \begin{center}
        \begin{tikzcd}
            \Proj^q_X A \arrow[d] \arrow[r] & X \arrow[d, "q"] \\
            \Proj_Y q_* A \arrow[r] & Y.
        \end{tikzcd}
    \end{center}
\end{definition}

The following is a direct generalization of \cite[Proposition~3.4]{EdidinRydh} to the derived setting.

\begin{theorem} \label{thm:gms of sat proj}
    Suppose that $q \colon X \to Y$ is a good moduli space and $A$ a quasi-coherent, graded $\oO_X$-algebra such that $\pi_0 A$ is finitely generated. Then $\Proj_X^q A \to \Proj_Y q_\ast A$ is a good moduli space and the induced morphism of good moduli spaces is the natural map $\Proj_Y q_\ast A \to Y$.
\end{theorem}

\begin{proof}
     Since we may work locally on $Y$ by Proposition~\ref{lem:gms_prop}(ii), we may assume that $Y$ is affine. Let $\epsilon \colon q^*q_*A \to A$ be the counit.
     
     For $f \in \pi_0(q_*A)$ homogeneous of positive degree, consider the following diagram with Cartesian squares
     \begin{center}
         \begin{tikzcd}
             \Spec_X A_{(\epsilon(q^*f))} \arrow[r] \arrow[d, "r_f"] & \Proj^q_X(A) \arrow[d, "r"] & \\
             \Spec_X (q^*q_*A)_{(q^*f)} \arrow[r] \arrow[d, "q'_f"] & \Proj_X(q^*q_*A) \arrow[d, "q'"] \arrow[r]  & X \arrow[d, "q"] \\
             \Spec (q_*A)_{(f)} \arrow[r]  & \Proj_Y(q_*A) \arrow[r] & Y
         \end{tikzcd}
     \end{center}
     where $q^*f \colon \oO_X[t] \to q^*q_*A$ is obtained from $f \colon \oO_Y[t] \to q_*A$ by applying $q^*(-)$.
     By Lemma \ref{lem:gms_prop}(ii) and Proposition \ref{Prop:Proj}, it suffices to show that $q'_fr_f$ is a good moduli space. 

     Note that we have a commutative diagram
     \begin{center}
         \begin{tikzcd}
             \oO_Y[t] \arrow[r, "\eta_{\oO_Y[t]}"] \arrow[d, "f"] & q_* \oO_X[t] \arrow[d, "q_*q^*f"] \\
             q_*A \arrow[r, "\eta_{q_*A} "] & q_*q^*q_*A \arrow[r, "q_*(\epsilon_A)"] & q_*A
         \end{tikzcd}
     \end{center}
     where $\eta$ is the unit of the adjunction $q^* \dashv q_*$. Now $q_*(\epsilon_A) \eta_{q_*A} \simeq \id_{q_*A} $. It follows that $(q_*A)_{(\epsilon(q^*f))} \simeq (q_*A)_{(f)}$, since $q_*$ commutes with taking the degree-zero part. Hence the claim follows from Lemma \ref{lem:gms_prop}(iii).
\end{proof}

We may now define saturated blow-ups along closed substacks. By the preceding theorem and the properties of Rees algebras established in \cite{Hekking}, they automatically admit good moduli spaces. See Appendix \ref{Sec:Blups} for a short overview on blow-ups and Rees algebras in derived algebraic geometry.

\begin{definition}
    Let $Z \to X$ be a closed immersion. The \emph{saturated blow-up} of $X$ along $Z$ is defined by $\Bl^q_Z X := \Proj_X^q \rR_{Z/X}$, where $\rR_{Z/X}$ denotes the Rees algebra associated to $Z\to X$. The \emph{exceptional divisor} of $\Bl^q_ZX$ is the restriction of the exceptional divisor $E_ZX \to \Bl_ZX$ of the blow-up of $X$ in $Z$ along the canonical inclusion $\Bl^q_ZX \subseteq \Bl_ZX$.
\end{definition}

We can describe the good moduli space of a saturated blow-up, which, when the closed substack is saturated, is particularly concrete.

\begin{proposition} \label{prop: push of Rees of sat}
    Let $Z \to X$ be a saturated closed immersion. We then have a natural equivalence of graded algebras $q_\ast \rR_{Z/X} \simeq \rR_{q(Z)/Y}$.
\end{proposition}

\begin{proof}
    In the commutative diagram
    \begin{center}
        \begin{tikzcd}
            Z \arrow[r] \arrow[d] & X \arrow[d, "q"] & D_{Z/X} \arrow[d, "{q'}"] \arrow[l, "p"] \\
            q(Z) \arrow[r] & Y & D_{q(Z)/Y}, \arrow[l, "r"]
        \end{tikzcd}
    \end{center}
    the square on the left is Cartesian, hence so is the one on the right. It follows that $q'$ is a good moduli space, and hence
    \[ q_*\rR_{Z/X}^{\extd} \simeq q_*p_* \sO_{D_{Z/X}} \simeq r_*q'_*\sO_{D_{Z/X}} \simeq r_* \sO_{D_{q(Z)/Y}} \simeq \rR_{q(Z)/Y}^{\extd}. \]
    Since pushforwards commute with the functor $(-)_{\geq 0}$ that sends a $\ZZ$-graded algebra to the $\NN$-graded algebra obtained by discarding the pieces in negative degree, the claim follows.     
\end{proof}




From here on, we work over $\CC$ for the remainder of this subsection. We consider the primary case of interest where $Z = X^{\max}$, defined in \cite[Theorem~4.14, Definition~4.15]{HRS}. $X^{\max}$ is a canonical, closed substack of $X$ which parametrizes points with stabilizers of maximal dimension. As for classical stacks, $X^{\max}$ does not have to be saturated. 

The saturated blow-up $\hat{X} = \Proj_X^q \rR_{X^{\max}/X}$ is called the Kirwan blow-up of $X$, defined in \cite[Definition~6.5]{HRS}. We obtain the following result based on the above.

\begin{theorem} \label{thm:gms of kirwan blow-up}
    Let $X$ be an Artin stack, locally of finite presentation with a good moduli space $q \colon X \to Y$, such that $X_\cl$ is Noetherian. Then its Kirwan blow-up $\hat{X} = \Proj_X^q \rR_{X^{\max}/X} $ admits a good moduli space $\hat{q} \colon \hat{X} \to \hat{Y}$, fitting in a commutative diagram
    \begin{align*}
        \xymatrix{
        \hat{X} \ar[d]_-{\hat{q}} \ar[r] & X \ar[d]^-{q} \\
        \hat{Y} \ar[r] & Y.
        }
    \end{align*}
    In the case when $X^{\max}$ is saturated, $\hat{Y} = \Proj_Y q_\ast \rR_{X^{\max}/X} =  \Bl_{q(X^{\max})} Y \to Y$ is the blow-up of $Y$ along the scheme-theoretic image of $X^{\max}$.
\end{theorem}

\begin{proof}
    Since $A = \rR_{X^{\max}/X}$ satisfies the condition that $\pi_0 A$ is finitely generated (and indeed in degree $1$!), Theorem~\ref{thm:gms of sat proj} applies and $\hat{X} = \Proj_X^q \rR_{X^{\max}/X}$ admits a good moduli space $\hat{Y} = \Proj_Y q_\ast \rR_{X^{\max}/X}$. Then Proposition~\ref{prop: push of Rees of sat} finishes the proof.
\end{proof}

Using the fact that the operation of Kirwan blow-up reduces the maximal stabilizer dimension of a stack by \cite[Theorem~6.7]{HRS}, the stabilizer reduction $\tilde{X}$ of $X$ is obtained by a finite sequence of iterated Kirwan blow-ups
$$ X_0 := X,\ X_1 = \hat{X}_0, \ldots,\ \tilde{X} = X_m = \hat{X}_{m-1}.$$

In particular, Theorem~\ref{thm:gms of kirwan blow-up} has the following implication.

\begin{corollary} \label{corr:gms of stab red}
    Let $X$ be a quasi-compact Artin stack, locally of finite presentation with a good moduli space $q \colon X \to Y$. Then the stabilizer reduction $\tilde{X}$ admits a good moduli space $\tilde{q} \colon \tilde{X} \to \tilde{Y}$.
\end{corollary}

\begin{remark}
    We expect there to be a stabilizer reduction of Artin stacks over a more general base than $\CC$ (in particular over $\ZZ$), and intend to come back to this in later work.
\end{remark}

\subsection{Geometric Invariant Theory (GIT) and tame stacks} We remark that an interesting special case of Theorem~\ref{Thm:GMSvsGMScl} and Proposition~\ref{Prop:comparison between Alper and derived GMS} covers Noetherian Artin stacks $X$ with quasi-affine diagonal whose classical truncation is obtained by GIT \cite{MFK}, i.e., is of the form $X_\cl = [Q^{ss} / G]$, where $Q^{ss} \subseteq Q$ is the locus of semistable points for a linearized action of a complex reductive group $G$ on a projective variety $Q$. This includes a plethora of stacks commonly encountered in the literature and provides natural derived enhancements of their corresponding GIT quotients. For instance, moduli spaces of sheaves (cf.~\cite{HuyLehn}) fall under this umbrella, with Example~\ref{example:moduli of sheaves} below being an explicit example.

Since moduli stacks of stable maps are Deligne--Mumford quotient stacks that can be constructed by GIT methods, we have the following example of this phenomenon.

\begin{theorem} \label{thm:3.12}
    Let $W$ be a complex projective variety, $\beta \in H_2(W,\ZZ)$ and $g,n \in \NN$. The derived Deligne--Mumford stack $\mathbb{R} \mM_{g,n}(W, \beta)$ of $n$-pointed stable maps of genus $g$ and class $\beta$ to $W$ (see \cite[Definition~2.6]{DerivedMgn}) admits a good moduli space $\mathbb{R} M_{g,n}(W, \beta)$, which is a derived enhancement of the classical good moduli space $M_{g,n}(W, \beta)$.
\end{theorem}

More generally, we say that an Artin stack $X$ is \emph{tame} if $X_\cl$ is tame, i.e., if $X_\cl$ has finite inertia and pointwise linearly reductive stabilizers. A \emph{tame moduli space} for $X$ is a good moduli space $q \colon X \to Y$ such that $q$ has finite diagonal (or equivalently quasi-finite and separated diagonal by~\cite[Theorem~8.3.2]{AlperAdequate}).

\begin{theorem}
    Let $X$ be a derived tame stack. Then $X$ admits a tame moduli space.
\end{theorem}

\begin{proof}
    Since $X_\cl$ is tame, it admits a good moduli space (see~\cite[Remark~4.2]{AlperGood} for a related discussion). Thus $X$ admits a good moduli space by Theorem~\ref{Thm:GMSvsGMScl}, which has finite diagonal and thus is tame.
\end{proof}

\begin{remark} Any Deligne--Mumford stack in characteristic $0$ with finite inertia is tame, hence admits a tame moduli space. In particular, this is the case for the stack $\mathbb{R} \mM_{g,n}(W, \beta)$ in Theorem~\ref{thm:3.12}.
\end{remark}

\subsection{Existence criteria} In the same vein as in the previous subsection, let $X$ be a Noetherian Artin stack with quasi-affine diagonal. By Theorem~\ref{Thm:GMSvsGMScl} and Proposition~\ref{Prop:comparison between Alper and derived GMS}, the existence of a good moduli space for $X$ is equivalent to the existence of an $\mathcal{A}$-good moduli space for $X_\cl$.

We may thus promote the existence criteria for $\mathcal{A}$-good moduli spaces of classical Artin stacks developed in \cite{AlperExistence} to the case of derived Artin stacks. Here is the translation of \cite[Theorem~A]{AlperExistence} to our context. Of course, one may obtain analogues of similar results, e.g., \cite[Theorem~4.1]{AlperExistence}.

\begin{theorem}
Let $X$ be a finitely presented Artin stack over $\CC$ with affine diagonal. Then $X$ admits a good moduli space $q \colon X \to Y$ with $Y$ separated if and only if $X_\cl$ is $\Theta$-reductive and $\mathsf{S}$-complete.

Moreover, $Y$ is proper if and only if $X_\cl$ satisfies the existence part of the valuative criterion for properness.
\end{theorem}

As an immediate application of this theorem, we obtain derived enhancements for good moduli spaces of moduli stacks parametrizing objects in abelian categories, when these stacks admit derived enhancements. 

Here is a concrete example.

\begin{theorem}
Let $W$ be a projective scheme over $\CC$ and $X$ be a quasi-compact component of the derived moduli stack of coherent sheaves on $W$. $X$ admits a good moduli space $q \colon X \to Y$, with $Y$ proper, which is a derived enhancement of the good moduli space $q_\cl \colon X_\cl \to Y_\cl$.
\end{theorem}

\begin{proof}
    By \cite[Lemma~7.20]{AlperExistence}, $X$ has affine diagonal and, by \cite[Theorem~7.23]{AlperExistence}, $X_\cl$ admits a good moduli space $q_\cl \colon X_\cl \to Y_\cl$ with $Y_\cl$ proper, so the claim follows by Theorem~\ref{Thm:GMSvsGMScl}.
\end{proof}

\subsection{Group actions and quotients} \label{subsection:group action examples}
Let $G$ be an affine group scheme which is flat and locally of finite presentation over a base scheme $S$. For a general definition of a group object in any $\infty$-topos and its classifying space, we encourage the reader to consult \cite{NikolausOOBundles}.
\begin{definition}
    We call $G$ \emph{linearly reductive} if $q \colon B_S G \to S$ is a good moduli space.
\end{definition}
From here on, assume that $G$ is linearly reductive. We would like to describe the good moduli space of affine quotients of $G$-actions.

\begin{example}
    Suppose that $G,S$ are classical. Since $B_SG$ is Noetherian and has affine diagonal, its good moduli space is an $\mathcal{A}$-good moduli space (Proposition \ref{Prop:comparison between Alper and derived GMS}). If $S = \Spec k$ for a field $k$, it follows that $G$ is linearly reductive, and reductive when $\chr k = 0$ (cf. \cite[Section~12]{AlperGood}). 
\end{example}

\begin{definition}
    Let $\QAlg^G(S)$ be the opposite of the category of schemes affine over $S$ with a $G$-action. 
\end{definition}

Observe that $\QAlg^G(S) \simeq \QAlg(BG)$. We obtain an adjunction
    \begin{center}
        \begin{tikzcd}
            \QAlg(S) \arrow[r, shift left, "q^*"]  & \QAlg^G(S) \arrow[l, shift left, "q_*"].
        \end{tikzcd}
    \end{center}   
Now $q^*$ endows $A \in \QAlg(S)$ with trivial $G$-action. It follows that $q_*$ is the derived invariants functor $A \mapsto A^G$. The counit induces a $G$-equivariant morphism $A^G \to A$.

Let $U = \Spec A$ be a scheme affine over $S$ with a $G$-action. Then $[U/G] \to \Spec_S q_*\oO_{[U/G]}$ is a good moduli space for $[U/G]$ by Lemma \ref{lem:gms_prop}, since $[U/G] \simeq \Spec_{B_SG} \oO_{[U/G]}$. The canonical map $\Spec q_*\oO_{[U/G]} \to U \git G := \Spec A^G$ is an equivalence, hence also $[U/G] \to \Spec A^G$ is a good moduli space.

\begin{remark}
    Suppose that $G$ is classical. We can give a hands-on description of $U \git G$ as follows. Using the description of $\Alg^G_{S} \coloneqq \QAlg^G(S)$ as the $\infty$-category associated to the 1-category of simplicial $G$-algebras endowed with the projective model structure, one can show that the adjunction 
    \[ i \dashv (-)^G \colon \Alg_{S} \rightleftarrows \Alg^G_{S} \]
    is induced by a Quillen adjunction between the respective model categories \cite[\S 3.1]{HRS}. Since $G$ is linearly reductive, we can thus compute $A^G$ by taking a simplicial resolution of $A$, and then taking classical invariants simplicial level-wise. 
    
    If $S = \Spec \CC$ and $U = \Spec A$ is of finite presentation, then a particular model can be obtained by starting with a $G$-invariant closed embedding $U_\cl \to \AA^N$, where $G$ acts on $\AA^N$ linearly. For each positive degree, we then add generators iteratively by adjoining $G$-representations. For more details, the reader can consult \cite[Section~3.7]{HRS} on this procedure. Taking classical $G$-invariants of these $G$-representations then induces a simplicial model for $A^G$.
\end{remark}

\begin{example} \label{example:moduli of sheaves} Here is a concrete example of the above. Consider the matrices
        \[
            X = \begin{pmatrix}
                x_{11} & x_{12} \\
                x_{21} & x_{22}
            \end{pmatrix}, \quad
            Y = \begin{pmatrix}
                y_{11} & y_{12} \\
                y_{21} & y_{22}
            \end{pmatrix}, \quad
            R = \begin{pmatrix}
                r_{11} & r_{12} \\
                r_{21} & r_{22}
            \end{pmatrix}\,,
        \]
    and let $\CC[X,Y]$ be the free $\CC$-algebra generated by the entries of $X$ and $Y$ and let $\CC[R]$ be the free $\CC$-algebra generated by the entries of $R$. Now let $\CC[R]\to \CC[X,Y]$ be the map sending $R$ to the commutator matrix $[X,Y] = XY-YX$. For example, $r_{12}$ maps to $[X,Y]_{12} = x_{11} y_{12} + x_{12} y_{22} - y_{11} x_{12} - y_{12} x_{22}$. Let $A$ be the $\CC$-algebra which is defined by the pushout diagram 
        \[
            \begin{tikzcd}
                \CC[R] \ar{d}\ar{r} & \CC[X,Y] \ar{d} \\
                \CC \ar{r} & A.
            \end{tikzcd}
        \]
        
    Then $G=\GL_{2,\CC}$ acts by simultaneous conjugation on $X,Y, R$ and hence on the derived affine scheme $\Spec A$. We thus have a derived good moduli space $[\Spec A/G]\to \Spec A^G$. One can see by direct computation that the ring $A^G$ is the pushout of the diagram  
        \[
            \begin{tikzcd}
                \CC[d,t] \ar{d}\ar{r} & (\CC[x,y]\otimes_{\CC} \CC[x,y])^{S_2} \ar{d}\\
                \CC \ar{r} & A^G,
            \end{tikzcd}
        \]
    where $d$ and $t$ are sent to zero via both maps and the symmetric group $S_2$ acts in the obvious way. This means that $\Spec A^G$ is a derived enhancement of the classical affine scheme $\Sym^2(\AA^2_\CC)$. 
    
    This should not come as a surprise: $[\Spec A / G]$ is the derived moduli stack of length $2$ sheaves on $\AA^2_\CC$ (for example, this follows by the same argument used in the proof of~\cite[Theorem~9]{KatzShi}), and $[\Spec A / G]_\cl$ is the classical moduli stack $\mM_2$ of length $2$ sheaves on $\AA^2_\CC$, whose good moduli space is given by the Hilbert--Chow morphism $\mM_2 \to \Sym^2(\AA^2_\CC)$.
\end{example}

If the base $S$ is classical, then $B_SG$ will in general not be algebraic unless $G$ is flat and locally of finite presentation, hence classical, as the following example illustrates.

\begin{example}  
Let $G = \Spec \CC[\epsilon]$ be the derived group scheme given by the derived dual numbers, i.e., the algebra $\CC[\epsilon] = \CC \oplus \CC\epsilon[1]$ where the group structure is given by considering $G$ as the loop stack of $\Spec \CC$. We model $\CC$-algebras by cdgas, using cohomological grading.

Suppose that $BG$ is algebraic. Then it admits a cotangent complex $\mathbb{L}_{BG}$. Moreover, since $(BG)_\cl = B(G_\cl) = \ast$ is affine, $BG$ is an affine scheme $BG = \Spec R$, where the cdga $R$ can be taken to be in standard form (see \cite[Def.~1.4]{HRS}) $[\ldots \to R^{-1} \xrightarrow{0} R^0=\CC]$ with $R^0 = \CC$, and the first differential is trivial.

The sequence $* \xrightarrow{u} BG \to *$ gives rise to an exact triangle
\begin{align*}
    u^\ast \BL_{BG} \lr \BL_{\ast} \lr \BL_{\ast / BG},
\end{align*}
so that $\BL_{\ast / BG} \simeq u^\ast \BL_{BG}[1]$. 
The Cartesian square
\begin{align} \label{eq:loc 4.2}
    \xymatrix{
    G \ar[r]^-{v} \ar[d]_-{v} & \ast \ar[d]^-{u} \\
    \ast \ar[r]_-{u} & BG,}
\end{align}
 gives an equivalence $\BL_G \simeq v^\ast \BL_{\ast/BG}$. We thus get an equivalence 
 \[ \BL_G \simeq v^\ast u^\ast \BL_{BG}[1]. \]
 By definition, $\BL_G \simeq \CC \cdot \mathrm{d} \epsilon [1]$ so that $h^{-1}(\BL_G) = \CC$. On the other hand, we see that $u^\ast \BL_{BG} = [ \ldots \to \Omega_{R^{-1}}|_\ast \xrightarrow{0} 0]$ and hence $h^{-1} \left( v^\ast u^\ast \BL_{BG}[1] \right) = 0$, which gives a contradiction. \smallskip
\end{example}


\appendix 

\section{Projective spectra} We call $\NN$-graded objects simply \emph{graded}, and refer the reader to \cite[\S 2]{Hekking} for background on graded objects in the derived setting. For a module $M$, we consider $\Sym(M)$ as a graded algebra which has $M$ in degree 1. For a graded algebra $A$ and positive integer $\delta$, write $A^{(\delta)}$ for the graded algebra such that $A^{(\delta)}_d = A_{\delta d}$. This definition extends to $\ZZ$-graded algebras/modules and integers $\delta$. A graded $k$-algebra $A$ is \emph{generated in degree 1} if the canonical map $\Sym_k(A_1) \to A$ is surjective. Recall that a line bundle is a quasi-coherent module which is locally free of rank 1.

\begin{remark}
\label{Rem:push_delta_twist}
    Let $q \colon X \to Y$ be a morphism of Artin stacks and $\delta \in \ZZ$. For $B$ a $\ZZ$-graded, quasi-coherent $\oO_Y$-algebra, it is clear that $q^*(B^{(\delta)}) \simeq (q^*B)^{(\delta)}$. We thus obtain, for any $\ZZ$-graded, quasi-coherent $\oO_X$-algebra $A$, a canonical map $(q_*A)^{(\delta)} \to q_*(A^{(\delta)})$. One shows this is an equivalence, by considering $\oO_Y$-linear maps $\oO_Y(-n) \to (q_*A)^{(\delta)}$, where $\oO_Y(-n)$ is the $\ZZ$-graded module which has $\oO_Y$ concentrated in degree $n$.
\end{remark}
 
 In \cite{Hekking}, the projective spectrum $\Proj A$ of a quasi-coherent, graded $\oO_X$-algebra $A$ generated in degree 1  is defined as
 \begin{equation}
 \label{Eq:Proj}
     \Proj  A \coloneqq [ (\Spec  A \setminus V(A_+)) / \GG_m ],
 \end{equation}
where $V(A_+) \to \Spec  A$ is the closed immersion associated to $A \to A_0$, where $A_0$ is the quasi-coherent $\oO_X$-algebra of degree-zero elements in $A$. In~\textit{loc.\ cit.,} $X$ is a scheme, but the same definition works for any stack $X$. It is shown that the projective spectrum is stable under base-change, and that $\Proj  A \to X$ is schematic. As in the classical case, $\Proj  A$ is covered by stacks of the form $\Spec  A_{(f)}$, where $A_{(f)}$ is the degree-zero part of $A_f$ and $f \in A$ is homogeneous of degree 1. In fact, the underlying classical stack of $\Proj A$ is the classical projective spectrum of $\pi_0A$.

If $A$ is not generated in degree 1, the formula \ref{Eq:Proj} gives a derived version of the stacky projective spectrum of $A$, which might no longer be schematic over the base \cite[\S 10.2.7]{Olsson}. Since for a good moduli space $q \colon X \to Y$ one might have that $q_*A$ is not generated in degree 1---even if $A$ is---we need to extend the definition of the projective spectrum to arbitrary graded algebras in order to make sense of the saturated $\Proj$. 

We do this as follows. Throughout, let $X$ be a stack, and $A$ a graded, quasi-coherent $\oO_X$-algebra.

\begin{definition}
    For $\delta \in \NN_{>0}$, let $\Proj ^{(\delta)}(A)$ be the following stack over $X$. For $g \colon T \to X$ with $T$ affine, let $\Proj ^{(\delta)}(A)(T)$ be the space of pairs $(L,\varphi)$, where $L$ is a line bundle on $T$ and $\varphi \colon g^*A^{(\delta)} \to \Sym(L)$ is a map of graded algebras which is surjective in degree 1. Naturality in $T$ is given by pulling back. It is clear that this is a stack, since graded algebras and their morphisms satisfy descent.
\end{definition}

Clearly, $\Proj^{(\delta)}(-)$ is stable under base-change. Moreover, for $\delta,\delta' \in \NN_{>0}$, we have a morphism
\[ i_{\delta,\delta'} \colon \Proj^{(\delta)} (A) \to \Proj^{(\delta\delta')} (A) \]
which on $g \colon T \to X$ sends $(L,\varphi)$ to $(L^{\otimes \delta'},\varphi^{(\delta')})$, where $\varphi^{(\delta')} \colon g^*A^{(\delta\delta')} \to \Sym(L^{\otimes \delta'})$ is determined by $\varphi$. Let $(\NN_{>0},\cdot)$ be the category associated to the poset on $\NN_{>0}$ determined by divisibility. We thus have a functor
\[ \Proj ^{(-)}(A) \colon (\NN_{>0},\cdot) \to \Stk_X  \]
into the category $\Stk_X$ of stacks over $X$.

\begin{lemma}
    \label{Lem:descend_linebs}
    The functor that sends an affine scheme $T$ to the space of line bundles on $T$ sends cofiltered limits to colimits. 
\end{lemma}

\begin{proof}
    This follows from the fact that $B\GG_m$ classifies line bundles and is locally of finite presentation in the sense of \cite[\S1]{Weil} or \cite[Def.~17.4.1.1]{LurieSpectral}.
\end{proof}

\begin{lemma}
\label{Lem:Projd}
    The morphism $i_{\delta,\delta'}$ is an open immersion.
\end{lemma}

\begin{proof}
    We sketch the following argument. One shows that $i_{\delta,\delta'}$ is formally \'{e}tale by using the lifting property against square-zero extensions $T \to T[M]$. Using Lemma \ref{Lem:descend_linebs}, one shows that $i_{\delta,\delta'}$ is locally of finite presentation. Finally, one shows that $i_{\delta,\delta'}$ is a monomorphism by showing that the diagonal of $i_{\delta,\delta'}$ is equivalent to the identity. 
\end{proof}

With the same argument as in the proof of \cite[Prop.~5.5.5]{Hekking}, we have a canonical map
\[  [ (\Spec  A \setminus V(A_+)) / \GG_m ] \to \Proj^{(1)}(A). \]

\begin{definition}
    The projective spectrum $\Proj (A)$ is the colimit of $\Proj ^{(-)}(A)$.
\end{definition}

\begin{proposition} Let $X,A$ be as above.
\label{Prop:Proj}
    \begin{enumerate}
        \item For any $f \colon X' \to X$, the map $\Proj_{X'}(f^*A) \to X' \times_X \Proj_X (A)$ is an equivalence.
        \item $\Proj (A)$ is covered by $\{\Spec  A_{(f)} \mid f \in A_d, d \geq 1 \}$. 
        \item Each $\Spec A_{(f)} \to \Proj (A)$ is an open immersion, and consequently $\Proj(A) \to X$ is schematic.
        \item If $A$ is generated in degree 1, then
        \[ \Proj^{(1)} (A) \simeq \Proj (A) \simeq [ (\Spec  A \setminus V(A_+)) / \GG_m ] \]
        \item The underlying classical stack of $\Proj A$ is the classical projective spectrum of $\pi_0A$ over $X_\cl$.        
        \item In general, the canonical map
        \[ [ (\Spec  A \setminus V(A_+)) / \GG_m ] \to \Proj (A) \]
        is a good moduli space morphism.
    \end{enumerate}
\end{proposition}

\begin{proof}

    (i) follows from the fact that colimits and each $\Proj^{(\delta)}(-)$ commute with base-change. We may thus assume that $X$ is affine in the remaining points, say $X = \Spec R$.

    For (ii), it suffices to show that $\Proj^{(\delta)}(A)$ is covered by schemes of the form $\Spec A_{(f)}$, where $f \in A$ is homogeneous of degree $\delta$. Let $\psi \colon A^{(\delta)} \to \Sym_{A_{(f)}}(A_{(f)}) = A_{(f)}[t]$ be the map of graded $R$-algebras which sends an element $a \in A^{(\delta)}_d = A_{\delta d}$ to $af^{- d}t^{d}$. The corresponding map $\varphi \colon A^{(\delta)} \otimes_R A_{(f)} \to A_{(f)}[t]$ of graded $A_{(f)}$-algebras is surjective in degree 1, and hence induces a morphism $\Spec A_{(f)} \to \Proj^{(\delta)}(A)$.

    To show that $\{ \Spec A_{(f)} \mid f \in A_\delta \}$ is a cover, for $g \colon T \to X$ with $T$ affine, let a $T$-point of $\Proj^{(\delta)}(A)$ be classified by $(L,\varphi)$. After taking a trivializing open cover of $T$, we may assume that $L = \oO_T$. Since $\varphi$ is surjective in degree 1, the $R$-linear map $\varphi' \colon A_\delta \to \oO_Tt = \oO_T$ induced by $\varphi$ must have a unit in its image. Taking $f \in A_\delta$ such that $\varphi'(f)$ is a unit induces a morphism $T \to \Spec A_{(f)}$ over $\Proj^{(\delta)}(A)$.

    For (iii), by Lemma \ref{Lem:Projd} it suffices to show that $\Spec A_{(f)} \to \Proj^{(\delta)}(A)$ is an open immersion, for each $f \in A_\delta$. By the previous argument, a lift of a $T$-point of $\Proj^{(\delta)}(A)$ to $\Spec A_{(f)}$ is unique, which shows that $\Spec A_{(f)} \to \Proj^{(\delta)}(A)$ is a formally \'{e}tale monomorphism. By Lemma \ref{Lem:descend_linebs}, one shows that $\Spec A_{(f)} \to \Proj^{(\delta)}(A)$ is locally of finite presentation.

    (iv) follows from \cite[Prop.~5.5.5]{Hekking}, and (v) from \cite[\href{https://stacks.math.columbia.edu/tag/01NS}{Tag 01NS}]{stacks-project}.

    Finally, for (vi) one shows that 
    \begin{center}
        \begin{tikzcd}
            {[\Spec A_f / \GG_m]} \arrow[r] \arrow[d] & {[ (\Spec A \setminus V(A_+)) / \GG_m ] }\arrow[d] \\
            \Spec A_{(f)} \arrow[r] & \Proj A
        \end{tikzcd}
    \end{center}
    is Cartesian, for each $f \in A$ homogeneous of positive degree. Hence the statement follows from Lemma \ref{lem:gms_prop}(ii) together with the description of good moduli spaces of affine group quotients in Subsection~\ref{subsection:group action examples}.
\end{proof}

\begin{corollary}
\label{Cor:Proj_nat}
    Let $\varphi \colon A \to B$ be a morphism of quasi-coherent, graded $\oO_X$-algebras. Then there is a largest open substack $U(\varphi)$ of $\Proj (B)$ on which $U(\varphi) \to \Proj(A)$ is well-defined. For $f \in A$ homogeneous of positive degree, we have a Cartesian diagram
    \begin{center}
        \begin{tikzcd}
            \Spec B_{(\varphi(f))} \arrow[d] \arrow[r] & U(\varphi) \arrow[d] \\
            \Spec A_{(f)} \arrow[r] & \Proj A.
        \end{tikzcd}
    \end{center}
    In particular, $U(\varphi) \to \Proj(A)$ is affine.
\end{corollary}

\begin{proof}
    This follows from the classical case by the previous proposition.
\end{proof}

\begin{example}
\label{Ex:Proj_delta_twist}
    For any positive integer $\delta$, there is a natural equivalence $\Proj A^{(\delta)} \to \Proj A$.
\end{example}

\section{Blow-ups}
\label{Sec:Blups}
We recall the theory of blow-ups in derived algebraic geometry, as developed in \cite{Hekking,Weil,BenBassatHekkingBlowups}. Endow $X \times \AA^1 = \Spec \oO_X[t^{-1}]$ with the $\GG_m$-action which corresponds to giving $t^{-1}$ degree $-1$.

Let $Z \to X$ be a closed immersion of stacks. The \emph{deformation space} is the $\GG_m$-equivariant stack $D_{Z/X}$ over $X \times \AA^1$ such that the space of $\GG_m$-equivariant $T$-points of $D_{Z/X}$ over $X \times \AA^1$ is equivalent to the space of $T \times_{X \times \AA^1} (X \times \{0\})$-points of $Z \times\{0\}$ over $X \times \{0\}$. Then $D_{Z/X}$ is affine over $X \times \AA^1$. We thus obtain a quasi-coherent, $\ZZ$-graded $\oO_X[t^{-1}]$-algebra $\rR_{Z/X}^{\extd}$---called the \emph{extended Rees algebra}---such that $D_{Z/X} \simeq \Spec \rR_{Z/X}^{\extd} $ as $\GG_m$-stacks over $X \times \AA^1$.

By discarding the part in negative degrees, we obtain the \emph{Rees algebra} $\rR_{Z/X}$ of $Z \to X$. Then the \emph{blow-up} of $X$ in $Z$ is the projective spectrum of the Rees algebra:
\[ \Bl_ZX \coloneqq \Proj \rR_{Z/X}. \]

Recall that a \emph{virtual Cartier divisor} is a quasi-smooth closed immersion $D \to T$ of virtual codimension 1. Thus, locally on $T$ it is of the form $\Spec R/(f) \to \Spec R$.

The \emph{conormal complex} of a morphism of Artin stacks $X \to Y$ is defined as the shifted cotangent complex $N_{X/Y}^\vee \coloneqq \BL_{X/Y}[-1]$.

\begin{definition}
    A \emph{strict virtual Cartier divisor} over $Z \to X$ is a commutative diagram
    \begin{center}
        \begin{tikzcd}
            D \arrow[r] \arrow[d, "g"] & T \arrow[d] \\
            Z \arrow[r] & X
        \end{tikzcd}
    \end{center}
    where $D \to T$ is a virtual Cartier divisor, such that $D_\cl \cong (T \times_X Z)_\cl$ and such that $g^*N_{Z/X}^\vee \to N_{D/T}^\vee$ is surjective.
\end{definition}

\begin{example}
    The \emph{exceptional divisor} of the blow-up is the projective bundle $E_ZX \coloneqq \PP_Z(N^\vee_{Z/X})$. There is a canonical map $E_ZX \to \Bl_ZX$ which induces a strict virtual Cartier divisor over $Z \to X$.
\end{example}

\begin{theorem}
    The blow-up $\Bl_ZX$ classifies strict virtual Cartier divisors over $Z \to X$, and $E_ZX \to \Bl_ZX$ is the universal strict virtual Cartier divisor over $Z \to X$.
\end{theorem}

We know that the extended Rees algebra functor 
\[ \rR_{(-)/X}^\extd \colon \mathsf{Cl}(X)^{\op} \to \QAlg^\ZZ(X \times \AA^1) \]
is a fully faithful left adjoint from the category of closed immersions $Z \to X$ into the category of $\ZZ$-graded, quasi-coherent $\oO_X[t^{-1}]$-algebras. The adjoint sends $Q \in \QAlg^\ZZ(X \times \AA^1)$ to $\Spec((Q/(t^{-1}))_0) \to X$.

\bibliographystyle{dary}
\bibliography{refs}	

\end{document}